\documentclass[12pt]{amsart}
\usepackage[margin=1in]{geometry}
\usepackage[dvipsnames,svgnames,x11names]{xcolor}
\usepackage[hyphens]{url}
\usepackage[colorlinks=true,
            linkcolor=Maroon,
            citecolor=blue,
            urlcolor=blue,
            hypertexnames=false,
            linktocpage]{hyperref}
\usepackage{bookmark}
\usepackage{amsmath,thmtools,mathtools,amssymb}
\usepackage{enumerate}
\usepackage{tikz}
\usetikzlibrary{arrows.meta,calc}
\mathtoolsset{showonlyrefs=true}
\allowdisplaybreaks
\newcommand{\R}{\mathbb{R}}
\newcommand{\bbN}{\mathbb{N}}
\newcommand{\bbR}{\mathbb{R}}
\newcommand{\bbS}{\mathbb{S}}
\newcommand{\bbB}{\mathbb{B}}
\newcommand{\Rn}{\mathbb{R}^{n+1}}
\newcommand{\Sn}{\mathbb{S}^n}
\newcommand{\Sd}{\mathbb{S}^{d-1}}
\newcommand{\Rplus}{\mathbb{R}^{n+1}_+}
\newcommand{\Cth}{\mathcal{C}_{\theta}}
\newcommand{\Hd}{\mathcal{H}}

\newcommand{\wh}{\widehat}
\newcommand{\ov}{\overline}
\newcommand{\sub}{\subset}
\newcommand{\op}[1]{\operatorname{#1}}
\newcommand{\mres}{\mathbin{\vrule height 1.6ex depth 0pt width 0.13ex\vrule height 0.13ex depth 0pt width 1.3ex}}

\newcommand{\al}{\alpha}
\newcommand{\de}{\delta}
\newcommand{\la}{\lambda}
\newcommand{\si}{\sigma}
\newcommand{\Si}{\Sigma}
\newcommand{\ze}{\zeta}
\newcommand{\Om}{\Omega}
\newcommand{\De}{\Delta}

\newcommand{\cB}{\mathcal{B}}
\newcommand{\cC}{\mathcal{C}}
\newcommand{\cF}{\mathcal{F}}
\newcommand{\cH}{\mathcal{H}}
\newcommand{\cK}{\mathcal{K}}
\newcommand{\cL}{\mathcal{L}}
\newcommand{\cO}{\mathcal{O}}
\newcommand{\cR}{\mathcal{R}}

\DeclareMathOperator{\id}{id}
\DeclareMathOperator{\supp}{supp}
\DeclareMathOperator{\dist}{dist}
\DeclareMathOperator{\tr}{tr}
\DeclareMathOperator{\vol}{vol}

\newcommand{\n}{\nabla}
\newcommand{\bn}{\nabla}
\newcommand{\bHess}{\nabla^2}

\newcommand{\ip}[2]{\langle #1,#2\rangle}
\newcommand{\abs}[1]{\lvert #1\rvert}

\theoremstyle{plain}
\newtheorem{theorem}{Theorem}[section]
\newtheorem{proposition}[theorem]{Proposition}
\newtheorem{lemma}[theorem]{Lemma}
\newtheorem{corollary}[theorem]{Corollary}

\theoremstyle{definition}
\newtheorem{definition}[theorem]{Definition}

\theoremstyle{remark}
\newtheorem{remark}[theorem]{Remark}

\numberwithin{equation}{section}

\newcommand{\eq}[1]{\begin{equation}\begin{alignedat}{2}#1\end{alignedat}\end{equation}}

\makeatletter
\renewcommand{\subsection}{\@startsection{subsection}{2}%
  \z@{.5\linespacing\@plus.7\linespacing}{.25\linespacing}%
  {\normalfont\bfseries}}
\makeatother

\title[Capillary $L_p$-Christoffel--Minkowski problem]
  {Capillary $L_p$-Christoffel--Minkowski problem}
\author[Y. Hu, M. N. Ivaki]{Yingxiang Hu, Mohammad N. Ivaki}

\begin{document}
\begin{abstract}
We solve the capillary $L_p$-Christoffel--Minkowski problem in the half-space for $1\leq k<n$ and $1<p<k+1$ in the class of even hypersurfaces. A crucial ingredient is a non-collapsing estimate that yields lower bounds for both the height and the capillary support function. Our result extends the capillary Christoffel--Minkowski existence result of \cite{HIS25}.
\end{abstract}
\maketitle

\section{Introduction}
The problem of prescribing area measures of convex hypersurfaces/bodies originates in the classical works of Christoffel \cite{Chr65}, Minkowski \cite{Min97,Min03}, Aleksandrov \cite{Ale56}, Nirenberg \cite{Nir57} and Pogorelov \cite{Pog52,Pog71}, which established the modern interplay between convex geometry and fully nonlinear elliptic equations. In the smooth setting, the Christoffel--Minkowski problem seeks a strictly convex hypersurface for which the $k$-th elementary symmetric function of the principal radii of curvature equals a prescribed function on $\bbS^n$. Further early contributions were made by Firey and Berg \cite{Fir67,Fir70,Ber69}.

For case $k=n$, corresponding to the classical Minkowski problem, Cheng--Yau \cite{CY76} and Caffarelli \cite{Caf90a,Caf90b} established an existence and regularity theory for the underlying fully nonlinear equation. For intermediate cases $1<k<n$, the theory is less complete, although Guan--Ma \cite{GM03} and Sheng--Trudinger--Wang \cite{STW04} provide existence results in the smooth setting. See also \cite{BHOM25,MU25} for recent advances in the rotationally symmetric case.

Within the Brunn--Minkowski--Firey theory initiated by Lutwak \cite{Lut93}, replacing classical area measures by their $L_p$ analogues leads to the $L_p$-Christoffel--Minkowski equation
\eq{
\si_k(\tau^{\sharp}[h])=h^{p-1}\phi\quad\text{on }\bbS^n
}
for the support function $h$ of a smooth, strictly convex body, where $\tau^{\sharp}[h]=g^{-1}\cdot(\nabla^2 h+hg)$ and $g$ is the round metric on $\bbS^n$. For $k=n$, the $L_p$-Minkowski problem has been studied extensively across a wide range of $p$; see \cite{LO95,CW06,BLYZ13,HLYZ16,BBCY19,HXZ21,GLW22,LXYZ24} and also \cite{CW00,BIS19,LWW20,CL21,BIS21b,BG23}. For $k<n$, much less is known for general intermediate $L_p$-area measures. In the smooth setting, however, existence, uniqueness, regularity, and constant rank theorems ensuring strict convexity are well developed, especially for $p> 1$ with even data on $\bbS^n$; see, for instance, \cite{GM03,HMS04,GLM06,GMZ06,GX18,Iva19,BIS23,BIS23b,HI24,Zha24,CH25} and \cite{BIS21,HLX24,LW24}.

In the presence of a boundary, one may consider capillary hypersurfaces in the half-space meeting a fixed supporting hyperplane at a prescribed contact angle $\theta\in(0,\pi/2)$. For the top-order case $k=n$, Mei, Wang and Weng \cite{MWW25a,MWW25b} solved the capillary $L_p$-Minkowski problem for $p\geq 1$ via the continuity method. For $-(n+1)<p<1$, \cite{HI25}  constructs even solutions through an iterative scheme based on a family of capillary curvature image operators, and in \cite{HHI25} a unified curvature flow approach  was later introduced for the even problem and all $p>-(n+1)$.

For $k<n$, the capillary analogue of the Christoffel--Minkowski problem prescribes $\si_k(\tau^{\sharp}[s])$ on the capillary spherical cap $\Cth$ and couples the interior equation with a Robin boundary condition encoding the contact angle. This capillary analogue was solved in \cite{HIS25}, where the existence of smooth, strictly convex, $\theta$-capillary hypersurfaces was established under conditions on the prescribed function that are tailored to the applicability of a constant rank theorem. The existence of a solution to the capillary Christoffel--Minkowski problem was also established in \cite{MWW25c}, subject to an additional assumption concerning the existence of a suitable homotopy path.

The aim of this paper is to extend the work \cite{HIS25} to the $L_p$-framework in the range $1<p<k+1$. In analogy with the closed case \cite{GM03,GX18}, we study the prescribed curvature equation
\eq{
\si_k(\tau^{\sharp}[s])=s^{p-1}\phi\quad\text{in }\Cth
}
for an even, positive, smooth function $\phi$ on $\Cth$, together with the capillary boundary condition
\eq{
\n_{\mu} s=\cot\theta\,s\quad\text{on }\partial\Cth.
}

\begin{theorem}\label{thm:main}
Let $1\leq k<n$, $1<p<k+1$, $\theta\in(0,\pi/2)$, and $\phi\in C^{\infty}(\mathcal{C}_{\theta})$ be a positive function satisfying
\eq{
\phi(-\zeta_1,\ldots,-\zeta_n,\zeta_{n+1})=\phi(\zeta_1,\ldots,\zeta_n,\zeta_{n+1})\quad \forall \zeta\in\Cth,
}
with
\eq{
\begin{dcases}
\nabla^2\phi^{-\frac{1}{p+k-1}}+g\phi^{-\frac{1}{p+k-1}}\geq 0 & \text{in }\Cth,\\
\nabla_{\mu}\phi^{-\frac{1}{p+k-1}}\leq \cot\theta\,\phi^{-\frac{1}{p+k-1}} & \text{on }\partial\Cth.
\end{dcases}
}
Then there exists a unique even, strictly convex, capillary hypersurface $\Sigma\subset\overline{\mathbb{R}^{n+1}_+}$ with contact angle $\theta$ whose capillary support function $s$ solves
\eq{
\label{capillary-Lp-eq}
\left\{
\begin{aligned}
\sigma_k(\tau^{\sharp}[s])&=s^{p-1} \phi&&\text{in }\Cth,\\
\nabla_{\mu} s&=\cot\theta\,s&&\text{on }\partial\Cth.
\end{aligned}
\right.
}
\end{theorem}

The paper is organized as follows. In Section \ref{sec:preliminaries} we recall the basic capillary geometry in the half-space and fix notation. In Section \ref{sec:non-collapsing} we establish non-collapsing estimates, namely, a lower bound for the height of the hypersurface, both in the rotationally symmetric and in the general even case. In Section \ref{sec:regularity} we derive curvature and regularity estimates for solutions of \eqref{capillary-Lp-eq}. In Section \ref{sec:strict-convexity} we prove a capillary constant rank theorem for our equation. Finally, in Section \ref{sec:existence-uniqueness} we complete the proof of \autoref{thm:main} by establishing existence and uniqueness.

\section{Preliminaries}\label{sec:preliminaries}
Let $\{e_i\}_{i=1}^{n+1}$ be the standard orthonormal basis of $\bbR^{n+1}$. Let
\eq{
\mathbb{R}^{n+1}_+=\{x\in\mathbb{R}^{n+1}: x_{n+1}> 0\}
}
be the upper half-space with boundary $\partial\mathbb{R}^{n+1}_+=\{x_{n+1}=0\}$. The unit ball of $\bbR^{n+1}$ is denoted by $\bbB$, and we write $\bbS^n$ for the unit sphere.

\begin{enumerate}[(1)]
\item \textbf{Support functions of convex bodies.} For a bounded convex set $K\subset\bbR^{n+1}$, the support function $h_K:\Sn\to \bbR$ is defined as
\eq{
h_K(u):=\sup\{\langle x,u\rangle:x\in K\},\quad u\in\Sn.
}
When no confusion can arise, we simply write $h:=h_K$.

\item \textbf{Area measures in $\bbR^{n+1}$.} Let $K\subset\bbR^{n+1}$ be a bounded convex set and $k\in\{0,\dots,n\}$. The $k$-th area measure $S_k(K,\cdot)$ is the finite Borel measure on $\Sn$ appearing in the classical local Steiner formula. If $K$ is smooth and strictly convex with principal radii of curvature $\la_1,\dots,\la_n$ at a point with outer unit normal $u\in\Sn$, then
\eq{
dS_k(K,u)=\frac{1}{\binom{n}{k}}\sigma_k(\la_1,\dots,\la_n)\,d\si(u),
}
where $d\si$ denotes the spherical Lebesgue measure on $\Sn$. If $L\subset\bbR^{n+1}$ is an $m$-dimensional linear subspace and $K\subset L$, we write $S_k^L(K,\cdot)$ for the $k$-th area measure of $K$ viewed as a convex set in $L$.

\item \textbf{Hausdorff measure and subspheres.} For any integer $d\geq 1$, we write $\Hd^d$ for the $d$-dimensional Hausdorff measure, and
\eq{
\bbS^d:=\{x\in\R^{d+1} : |x|=1\}
}
for the unit sphere in $\R^{d+1}$. For a linear subspace $L\subset\R^{n+1}$ of dimension $m$, we identify $\Sn\cap L$ with the unit sphere in $L$. We also write
\eq{
\mathbb{S}^n_{\theta}:=\{x\in\Sn : \langle x,e_{n+1}\rangle\geq \cos\theta\}, \quad \cC_{\theta}:=\bbS^n_{\theta}-\cos\theta\, e_{n+1}.
}

Integrals of the form
\eq{
\int_{\Sn} f,\quad\int_{\mathbb{S}^n_{\theta}} f,\quad\int_{\Cth} f,\quad\int_{\Sn\cap L} f,\quad\int_{\bbS^d} f,
}
are always understood with respect to the restriction of the appropriate Hausdorff measure (thus, $\Hd^n$ on $\Sn$, $\mathbb{S}^n_{\theta}$ and $\Cth$, $\Hd^{m-1}$ on $\Sn\cap L$, and $\Hd^d$ on $\bbS^d$). We also write
\eq{
\omega_d:=\Hd^d(\mathbb{S}^d),
}
so that $\omega_d$ is the surface area of $\mathbb{S}^d$.
\end{enumerate}

\begin{definition}
A smooth, compact, connected, orientable hypersurface $\Sigma\subset\overline{\mathbb{R}^{n+1}_+}$ with nonempty boundary $\partial\Sigma\subset\partial\mathbb{R}^{n+1}_+$ and interior $\op{int}(\Sigma)\subset\mathbb{R}^{n+1}_+$ is called a capillary hypersurface with contact angle $\theta\in(0,\pi)$ if
\eq{
\langle\nu, e_{n+1}\rangle=\cos\theta \quad \text{on } \partial\Sigma,
}
where $\nu$ is the outer unit normal of $\Sigma$.
\end{definition}

The model capillary surface is
\eq{
\mathcal{C}_{\theta}=\{\zeta\in\overline{\mathbb{R}^{n+1}_+} : |\zeta+\cos\theta\, e_{n+1}|=1\}.
}
Via the translation $T(\zeta):=\zeta+\cos\theta\,e_{n+1}$, we may identify \(\mathcal{C}_{\theta}\) with \(\mathbb{S}^n_{\theta}\).

We also define
\eq{
\mathcal{C}_{\theta, r}:=\left\{\zeta\in\overline{\mathbb{R}_+^{n+1}} : |\zeta+r \cos \theta\, e_{n+1}|=r\right\}.
}
Note that the radius of $\partial \mathcal{C}_{\theta, r}$ is $r\sin\theta$.

We call $\Si$ strictly convex if the enclosed region $\widehat{\Si}$ is a convex body (i.e. compact, convex, with non-empty interior) and the second fundamental form of $\Si$ is positive definite. For a strictly convex capillary hypersurface $\Sigma$, the capillary Gauss map is defined as
\eq{
\tilde{\nu}=\nu-\cos\theta\, e_{n+1}: \Sigma\to \Cth.
}
The map $\tilde{\nu}$ is a diffeomorphism onto $\Cth$; see \cite[Lem. 2.2]{MWWX25}.

\begin{definition}
Let $\Si$ be a strictly convex, capillary hypersurface. The capillary support function $s:\Cth\to \mathbb{R}$ of $\Sigma$ is defined by
\eq{
s(\zeta)=\langle\tilde{\nu}^{-1}(\zeta), \zeta+\cos\theta\, e_{n+1}\rangle.
}
\end{definition}

For the model cap $\mathcal{C}_{\theta}$, the capillary support function is
\eq{
\ell(\zeta)=\sin^2\theta-\cos\theta\langle\zeta, e_{n+1}\rangle.
}

On $\Cth$ we write $g$ for the round metric, $\nabla$ for its Levi-Civita connection and $\nabla^2$ for the covariant Hessian. For $f\in C^2(\Cth)$, define
\eq{
\tau[f]:=\nabla^2 f+fg,\quad
\tau^{\sharp}[f]:=g^{-1}\cdot\tau[f].
}
Then $\tau^{\sharp}[f]$ is a self-adjoint endomorphism. Denoting its eigenvalues by $\la_1,\dots,\la_n$, write
\eq{
\sigma_k(\tau^{\sharp}[f]):=\sigma_k(\la_1,\dots,\la_n).
}
Moreover, $\nabla_{\mu}f$ denotes the covariant derivative of $f$ in the direction of the outward unit conormal $\mu$ along $\partial\Cth$.

For a symmetric matrix $A=(a_{ij})$, we use the notation
\eq{
\sigma_k^{ij}(A):=\frac{\partial\sigma_k}{\partial a_{ij}}(A).
}
Similarly, for $F=\sigma_k^{1/k}$, we set
\eq{
F^{ij}(A):=\frac{\partial F}{\partial a_{ij}}(A)
=\frac{1}{k}\,\sigma_k(A)^{\frac{1}{k}-1}\sigma_k^{ij}(A).
}

When $A=\tau^{\sharp}[s]$ for some function $s$ on $\Cth$, we abbreviate $\sigma_k^{ij}(\tau^{\sharp}[s])$ and $F^{ij}(\tau^{\sharp}[s])$ by $\sigma_k^{ij}$ and $F^{ij}$, respectively.

Writing points of $\Rn$ as $x=(x_1,\dots,x_n,x_{n+1})$, let $\cR$ denote the reflection
\eq{
\cR(x_1,\dots,x_n,x_{n+1}):=(-x_1,\dots,-x_n,x_{n+1}).
}
A function $\varphi:\Cth\to \bbR$ is called even if $\varphi\circ\cR=\varphi$, and we say that a capillary hypersurface $\Si$ (or its capillary support function $s$) is even if
\eq{
x\in\Si\implies\cR(x)\in\Si.
}

\begin{definition}\label{def:mixed-volume}
Let $s_0,\dots,s_n\in C^{\infty}(\Cth)$ be capillary support functions. Denote by $Q$ the polarization of $\sigma_n$ on symmetric endomorphisms of $T\Cth$, i.e. the unique symmetric $n$-linear map such that
\eq{
Q(A,\dots,A)=\sigma_n(A)\quad\text{for every symmetric endomorphism }A.
}
The capillary mixed volume of $s_0,\dots,s_n$ is defined by
\eq{
V(s_0,s_1,\dots,s_n):=\frac{1}{n+1}\int_{\Cth}s_0Q\bigl(\tau^{\sharp}[s_1],\dots,\tau^{\sharp}[s_n]\bigr).
}
Since $\tau^{\sharp}[\ell]=\id$, for every $k\in\{0,\dots,n\}$ one has
\eq{
V(s_0,\underbrace{s,\dots,s}_{k\text{-times}},\underbrace{\ell,\dots,\ell}_{(n-k)\text{-times}})=\frac{1}{n+1}\int_{\Cth} s_0\frac{\sigma_k\bigl(\tau^{\sharp}[s]\bigr)}{\binom{n}{k}}.
}
\end{definition}

\begin{theorem}\label{lem:tilde-nu-parallel body}
Let $\Sigma\subset\ov{\bbR^{n+1}_+}$ be a strictly convex, $\theta$-capillary hypersurface. For $t>0$ define
\eq{
\phi_t:\Sigma\to \ov{\bbR^{n+1}_+},\quad \phi_t(x):=x+t \tilde{\nu}(x).
}
Then $\Sigma_t:=\phi_t(\Sigma)$ is a strictly convex, $\theta$-capillary hypersurface. Moreover, the (standard) outer parallel convex body
\eq{
K_t:=\bigl(\widehat{\Sigma}-t\cos\theta\,e_{n+1}\bigr)+t\bbB,
}
and the capillary outer parallel convex body are related via $\wh{\Si_t}=K_t\cap \ov{\bbR^{n+1}_+}$. In addition, we have $\widehat{\Si_t}=\widehat{\Si}+t\widehat{\Cth}$.
\end{theorem}

\begin{proof}
Let $P=\{x_{n+1}=0\}$. Define $f(x)=\langle x,e_{n+1}\rangle=x_{n+1}$ and
\eq{
\tilde{f}(x)=\langle\tilde{\nu}(x),e_{n+1}\rangle=\langle\nu(x),e_{n+1}\rangle-\cos\theta.
}
Then
\eq{
\langle\phi_t(x),e_{n+1}\rangle=f(x)+t\tilde{f}(x).
}

\emph{Step 1.} On $\partial\Sigma$ we have $f=0$ and $\tilde{f}=0$. Hence $(f+t\tilde{f})(x)=0$ for $x\in\partial\Sigma$, i.e. $\phi_t(\partial\Sigma)\subset P$. Moreover, since $\nu(\op{int}(\Si))\subset \op{int}(\bbS^n_{\theta})$, we have $f+t\tilde{f}> 0$ on $\op{int}(\Sigma)$ for any $t>0$.

\emph{Step 2.} Let $D$ denote the Euclidean connection. At $x\in\Sigma$, choose an orthonormal basis $\{v_1,\dots,v_n\}$ of principal directions in $T_x\Sigma$. Then
\eq{
D_{v_i}\nu=\kappa_i v_i,\quad i=1,\dots,n,
}
where $\kappa_i$ are the principal curvatures. Since $\tilde{\nu}-\nu=-\cos\theta\,e_{n+1}$ is constant, differentiating $\phi_t$ yields
\eq{
d\phi_t(v_i)=v_i+tD_{v_i}\tilde{\nu}=v_i+tD_{v_i}\nu=(1+t\kappa_i)v_i.
}
Since $1+t\kappa_i> 0$ for each $i=1,\dots,n$, the image $d\phi_t(T_x\Sigma)$ is spanned by $\{v_1,\dots,v_n\}$, $\phi_t$ is a smooth immersion, and the unit normal of $\Sigma_t=\phi_t(\Sigma)$ at $y=\phi_t(x)$ induced by the orientation of $\Sigma$ equals $\nu(x)$, i.e.
\eq{
\nu_t(y)=\nu(x)\quad\text{for }y=\phi_t(x).
}
Next we show that $\phi_t$ is an injective immersion and thus an embedding. Assume for some $x,x'\in\Sigma$ we have $\phi_t(x)=\phi_t(x')$ for some $x,x'\in\Sigma$. Then
\eq{
x+t(\nu(x)-\cos\theta\,e_{n+1})=x'+t(\nu(x')-\cos\theta\,e_{n+1}),
}
which rearranges to
\eq{
x-x'=t\bigl(\nu(x')-\nu(x)\bigr).
}
Taking the inner product with $\nu(x)$, we obtain
\eq{
\label{eq:inj-ip}
\ip{x-x'}{\nu(x)}=t\bigl(\ip{\nu(x')}{\nu(x)}-1\bigr).
}
Since $\wh{\Si}$ is convex and $\nu(x)$ is the outer normal vector to $\Si$ at $x$,
\eq{
\ip{x'-x}{\nu(x)}\leq 0.
}
On the other hand, $\ip{\nu(x')}{\nu(x)}\leq 1$. By \eqref{eq:inj-ip}, we obtain
\eq{
\ip{x-x'}{\nu(x)}=0\quad\text{and}\quad\ip{\nu(x')}{\nu(x)}=1.
}
Thus $\nu(x')=\nu(x)$. Since $\Sigma$ is strictly convex, the Gauss map $\nu:\Sigma\to \bbS_{\theta}^n$ is injective, and $x'=x$.

\emph{Step 3.} If $y=\phi_t(x)$ with $x\in\partial\Sigma$, then by Step 1 and Step 2 we have $y\in P$ and $\nu_t(y)=\nu(x)$. Therefore,
\eq{
\langle\nu_t(y),e_{n+1}\rangle=\langle\nu(x),e_{n+1}\rangle=\cos\theta,
}
and $\Sigma_t$ meets $P$ with the same contact angle $\theta$.

\emph{Step 4.} Let $x\in\Sigma$ and set
\eq{
y:=\phi_t(x)=x+t(\nu(x)-\cos\theta\,e_{n+1}).
}
We claim that $y\in\partial K_t$ and that $\nu(x)$ is an outer normal of $K_t$ at $y$.

Note that $y\in K_t$. Since $\nu(x)$ is an outer unit normal of the convex body $\widehat{\Sigma}$ at $x$, we have
\eq{
\label{eq:support-K}
\ip{z-x}{\nu(x)}\leq 0\quad\forall\,z\in\widehat{\Sigma}.
}
Let $w\in K_t$. Then for some $z\in\widehat{\Sigma}$ and $b\in\bbB$:
\eq{
w=(z-t\cos\theta\,e_{n+1})+tb.
}
Moreover, we have
\eq{
\ip{w-y}{\nu(x)}=\ip{z-x}{\nu(x)}+t\ip{b}{\nu(x)}-t.
}
Using \eqref{eq:support-K}, we obtain
\eq{
\ip{w-y}{\nu(x)}\leq 0\quad\forall\,w\in K_t.
}
Hence we must have $y\in\partial K_t$ and $\nu(x)$ is an outer normal of $K_t$ at $y$.

\emph{Step 5.} Let $L_t=K_t\cap\ov{\bbR^{n+1}_+}$. We prove
\eq{
\Sigma_t=\phi_t(\Sigma)\subset\partial L_t.
}

If $x\in\op{int}(\Sigma)$, then by Step 1, we have
\eq{
y=\phi_t(x)\in\bbR^{n+1}_+.
}
By Step 4, $y\in\partial K_t$, and thus $y\in\partial L_t\setminus P$.

If $x\in\partial\Sigma$, then by Step 1, $\phi_t(\partial\Sigma)\subset P$, which implies $y\in P$. Let $x_j\in\op{int}(\Sigma)$ be any sequence with $x_j\to x$. Set $y_j:=\phi_t(x_j)$. By continuity, $y_j\to y$. Since $y_j\in\partial L_t$, the limit point $y$ belongs to $\partial L_t\cap P$.

\emph{Step 6.} We prove $\ov{\partial L_t\setminus P}=\Sigma_t$. By Steps 1, 2 and 5,
\eq{
\op{int}(\Si_t)=\phi_t(\op{int}(\Si))\subset\partial L_t\setminus P\implies\Sigma_t\subset\ov{\partial L_t\setminus P}.
}
It remains to prove $\partial L_t\setminus P\subset\Sigma_t\setminus P$.

Let $y\in\partial L_t\setminus P$. Then $y\in\partial K_t$. Suppose $u\in\bbS^n$ is an outer unit normal to $K_t$ at $y$, i.e.
\eq{
\label{eq:supporting-at-y}
\ip{w-y}{u}\leq 0\quad\forall\,w\in K_t.
}
We may write
\eq{
\label{eq:y-decomposition}
y=(x-t\cos\theta\,e_{n+1})+tb,\quad x\in\widehat{\Sigma},\quad b\in\bbB.
}
We claim that $b=u$, $x\in\partial \wh{\Si}$, and $u$ is an outer normal of $\widehat{\Sigma}$ at $x$.

Indeed, take any $x_0\in\widehat{\Sigma}$ and any $c\in\bbB$, and set
\eq{
w=(x_0-t\cos\theta\,e_{n+1})+tc\in K_t.
}
Substituting this $w$ and \eqref{eq:y-decomposition} into \eqref{eq:supporting-at-y}, we obtain
\eq{
0\geq \ip{w-y}{u}=\ip{x_0-x}{u}+t\ip{c-b}{u}.
}
With $x_0=x$ and $c=u$,
\eq{
1\leq \ip{b}{u}\implies b=u.
}
With $c=b=u$, we obtain
\eq{
0\geq \ip{x_0-x}{u}\quad\forall\,x_0\in\widehat{\Sigma}.
}
Thus $x\in\partial \wh{\Si}$ and $u$ is an outer normal of $\widehat{\Sigma}$ at $x$.

Since $y_{n+1}> 0$ and $t>0$, we have from \eqref{eq:y-decomposition} (with $b=u$)
\eq{
y_{n+1}=x_{n+1}-t\cos\theta+tu_{n+1}> 0.
}
Hence $x_{n+1}> 0$, $x\in\Si$, and $u=\nu(x)$ (if $x_{n+1}=0$, then $u_{n+1}\leq \cos\theta$, contradicting $y_{n+1}> 0$). Substituting $b=u=\nu(x)$ into \eqref{eq:y-decomposition} yields
\eq{
y=x+t(\nu(x)-\cos\theta\,e_{n+1})=\phi_t(x)\in\Sigma_t\setminus P.
}

\emph{Step 7.} By the previous steps, $\Si_t$ is a strictly convex (i.e. the enclosed region $\wh{\Si_t}$ is a convex body and the second fundamental form of $\Si_t$ is positive definite), $\theta$-capillary hypersurface, and for each point $x\in\Sigma$, the outward unit normal at the point $\phi_t(x)\in\Sigma_t$ is $\nu(x)$. Let $\zeta=\nu(x)-\cos\theta\, e_{n+1}$. Then
\eq{
s_{\Sigma_t}(\zeta)&=\langle\phi_t(x),\nu(x)\rangle\\
&=\langle x+t(\nu(x)-\cos\theta\, e_{n+1}),\nu(x)\rangle\\
&=s_{\Sigma}(\zeta)+t\ell(\zeta).
}
By \cite[Prop. 2.8]{MWWX25}, the Minkowski sum $\widehat{\Sigma}+t\widehat{\Cth}$ is a capillary convex body with capillary support function $s_{\Sigma}+t\ell$. Since capillary support functions determine capillary convex bodies, we conclude that $\widehat{\Sigma_t}=\widehat{\Sigma}+t\widehat{\Cth}$.
\end{proof}

\begin{remark}
The notion of local outer parallel sets for capillary convex bodies was first introduced in \cite{MWW25c}, while $\widehat{\Sigma_t}=\widehat{\Sigma}+t\widehat{\Cth}$ was observed in \cite[Rem. 2.17]{MWWX25}. \autoref{lem:tilde-nu-parallel body} clarifies the connection between capillary and classical outer parallel hypersurfaces; see \autoref{fig:example}.
\end{remark}
\begin{figure}
\begin{center}
\begin{tikzpicture}[
scale=4,
semithick,
font=\small,
plane/.style={black},
auxplane/.style={black, dashed},
prored/.style={red!70!black},
problue/.style={blue!70!black},
progreen/.style={green!50!black, dotted}
]

\pgfmathsetmacro{\tval}{0.25}
\pgfmathsetmacro{\cth}{sqrt(2)/2}          
\pgfmathsetmacro{\h}{\tval*\cth}           

\pgfmathsetmacro{\r}{1.0}
\pgfmathsetmacro{\R}{\r+\tval}
\pgfmathsetmacro{\a}{\r*sqrt(2)/2}

\coordinate (Cblk) at (0,-\a);
\coordinate (Cred) at (0,{-\a-\h});

\draw[plane] (-1.20,0) -- (1.50,0);
\node[anchor=west] at (1.58,0) {$x_{n+1}=0$};

\draw[auxplane] (-1.20,-\h) -- (1.50,-\h);
\node[anchor=west] at (1.58,-\h) {$x_{n+1}=-t\cos\theta$};

\draw[black] (Cblk) ++(45:\r) arc (45:135:\r);

\draw[prored] (Cred) ++(45:\r) arc (45:135:\r);

\draw[problue] (Cred) ++(45:\R) arc (45:135:\R);

\coordinate (Rr) at (\a,-\h);
\coordinate (Rl) at (-\a,-\h);

\draw[progreen] (-\a,{-\h-\tval}) -- (\a,{-\h-\tval});
\draw[progreen] (Rr) ++(-90:\tval) arc (-90:45:\tval);
\draw[progreen] (Rl) ++(135:\tval) arc (135:270:\tval);

\begin{scope}[
-{Stealth[scale=0.9]},
draw=gray,
dotted
]
\foreach \ang in {45,55,...,135} {
\draw ($(Cred)+(\ang:\r)$) -- ($(Cred)+(\ang:\R)$);
}

\foreach \xx in {-0.60,-0.30,0,0.30,0.60} {
\draw (\xx,-\h) -- (\xx,{-\h-\tval});
}

\foreach \ang in {-90,-75,...,45} {
\draw (Rr) -- ($(Rr)+(\ang:\tval)$);
}
\foreach \ang in {135,150,...,270} {
\draw (Rl) -- ($(Rl)+(\ang:\tval)$);
}
\end{scope}

\coordinate (LblBluePt) at ($(Cred)+(70:\R)$);
\node[text=blue!70!black, anchor=west] at ($(LblBluePt)+(0.03,0.02)$) {$\Sigma_t$};

\coordinate (TopBlk) at ($(Cblk)+(90:\r)$);
\node[text=black, anchor=north] at ($(TopBlk)+(0.18,0.00)$) {$\Sigma$};

\node[text=red!70!black, anchor=north] at (0.15,0.00)
{$\Sigma-t\cos\theta\,e_{n+1}$};

\end{tikzpicture}
\caption{Capillary vs. classical outer parallel hypersurfaces}
\label{fig:example}
\end{center}
\end{figure}
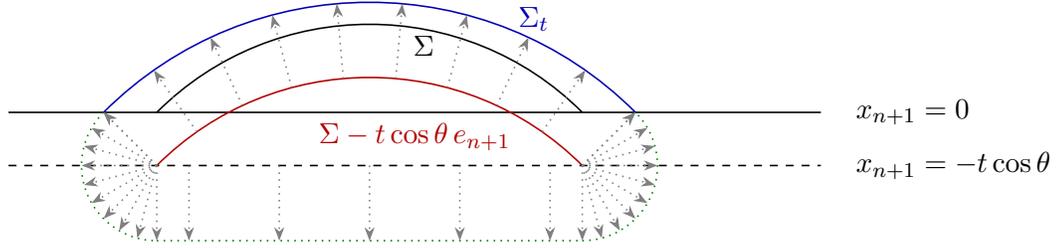

For $\rho> 0$ and a Borel set $\omega\sub\Cth$, the local outer parallel set of $\wh{\Si}$ in the directions of $\omega$ can be defined by
\eq{
\label{eq:def-B-rho}
B_{\rho,\theta}(\wh{\Si},\omega)&:=\left\{y\in\ov{\Rplus}:
\begin{array}{l}
\exists\, x\in\Si,\ 0<t<\rho,\ \text{such that}\\
y=x+t\tilde{\nu}(x),\ \tilde{\nu}(x)\in\omega
\end{array}
\right\}.
}

\begin{lemma}\label{lem:local-Steiner}
Let $\Si\sub\ov{\Rplus}$ be a strictly convex $\theta$-capillary hypersurface with principal curvatures $\kappa=(\kappa_1,\dots,\kappa_n)$ and area element $d\mu$. Then, for every Borel set $\omega\sub\Cth$ and every $\rho> 0$,
\eq{
\vol\bigl(B_{\rho,\theta}(\wh{\Si},\omega)\bigr)=\sum_{j=0}^n \frac{\rho^{n+1-j}}{n+1-j}\int_{\Si\cap\tilde{\nu}^{-1}(\omega)}\bigl(1-\cos\theta\ip{\nu}{e_{n+1}}\bigr)\si_{n-j}(\kappa)\,d\mu.
}
\end{lemma}
\begin{proof}
The local Steiner-type formula was previously stated in \cite{MWW25c}. For completeness, we include a proof here. Let
\eq{
\Phi:\Si\times(0,\infty)\to \ov{\Rplus},\quad\Phi(x,t):=x+t\tilde{\nu}(x).
}
By \autoref{lem:tilde-nu-parallel body}, for each $t>0$ the map $\Phi(\cdot,t)$ embeds $\Si$ onto the strictly convex, $\theta$-capillary hypersurface $\Si_t$. Since $s_{\Si_{t'}}=s_{\Si_t}+(t'-t)\ell>s_{\Si_t}$ for $t'>t$, the hypersurfaces $\Si_t$ are pairwise disjoint, and thus $\Phi$ is injective on $\Si\times(0,\infty)$.

For a given Borel set $\omega\sub\Cth$, by \eqref{eq:def-B-rho} we have
\eq{
B_{\rho,\theta}(\wh{\Si},\omega)=\Phi\Bigl(\bigl(\Si\cap\tilde{\nu}^{-1}(\omega)\bigr)\times(0,\rho)\Bigr).
}
Hence
\eq{
\label{eq:vol-Jac}
\vol\bigl(B_{\rho,\theta}(\wh{\Si},\omega)\bigr)=\int_{\Si\cap\tilde{\nu}^{-1}(\omega)}\int_0^{\rho} J(x,t)\,dt\,d\mu(x),
}
where $J(x,t)$ denotes the Jacobian of $\Phi$ at $(x,t)$.

Set $e=-e_{n+1}$. We have
\eq{
J(x,t)&=\ip{\tilde{\nu}(x)}{\nu(x)}\prod_{i=1}^n(1+t\kappa_i(x))\\
&=\bigl(1+\cos\theta\,\ip{\nu(x)}{e}\bigr)\prod_{i=1}^n(1+t\kappa_i(x))\\
&=\bigl(1+\cos\theta\,\ip{\nu(x)}{e}\bigr)\sum_{j=0}^n \si_{n-j}(\kappa(x))t^{n-j}.
}
Inserting this into \eqref{eq:vol-Jac} and integrating in $t$ yields
\eq{
\vol\bigl(B_{\rho,\theta}(\wh{\Si},\omega)\bigr)&=\int_{\Si\cap\tilde{\nu}^{-1}(\omega)}\int_0^{\rho}\bigl(1+\cos\theta\ip{\nu}{e}\bigr)\sum_{j=0}^n \si_{n-j}(\kappa)t^{n-j}\,dt\,d\mu\\
&=\sum_{j=0}^n\frac{\rho^{n+1-j}}{n+1-j}\int_{\Si\cap\tilde{\nu}^{-1}(\omega)}\bigl(1+\cos\theta\ip{\nu}{e}\bigr)\si_{n-j}(\kappa)\,d\mu.
}
\end{proof}

\begin{definition}\label{capillary-k-area-measure}
Let $k\in\{0,\dots,n\}$, $\theta\in(0,\pi/2)$, and suppose $\Si\subset\ov{\R^{n+1}_+}$ is a strictly convex, $\theta$-capillary hypersurface. For a Borel set $\omega\subset\Cth$, the capillary $k$-th area measure of $\wh{\Si}$ over $\omega$ can be defined by (see also \cite{MWW25c})
\eq{
S_{k,\theta}(\wh{\Si},\omega):=\binom{n}{k}^{-1}\int_{\tilde{\nu}^{-1}(\omega)}(1-\cos\theta\,\langle\nu,e_{n+1}\rangle)\sigma_{n-k}(\kappa)\,d\mu.
}

The capillary $k$-th area measure $S_{k,\theta}(\wh{\Si},\cdot)$ is absolutely continuous with respect to the $n$-dimensional Hausdorff measure $\cH^n\mres\Cth$, with density
\eq{
dS_{k,\theta}(\wh{\Si},\xi)=\binom{n}{k}^{-1}\ell(\xi)\sigma_k\bigl(\tau^{\sharp}[s](\xi)\bigr)\,d\cH^n(\xi),\quad \xi\in\Cth.
}
In particular,
\eq{
S_{k,\theta}(\wh{\Si},\omega)=\binom{n}{k}^{-1}\int_{\omega} \ell(\xi)\sigma_k\bigl(\tau^{\sharp}[s](\xi)\bigr)\,d\cH^n(\xi),\quad \omega\subset\Cth\ \text{Borel}.
}
\end{definition}

\begin{remark}\label{area measure-standard-vs-capillary}
The capillary $k$-th area measure $S_{k,\theta}(\wh{\Sigma},\cdot)$ is defined on $\Cth$ via the local Steiner formula and is absolutely continuous with respect to spherical Lebesgue measure on $\Cth$. Hence, it vanishes on every Borel set $\omega\subset\partial\Cth$. On $\op{int}(\Cth)$, it is the $\ell$-weighted pullback under $T$ of the restriction of the classical $k$-th area measure $S_k(\wh{\Sigma},\cdot)$ to $\op{int}(\bbS^n_{\theta})$.

The classical measure $S_k(\wh{\Sigma},\cdot)$ is defined on the whole sphere and pertains to the closed convex body $\wh{\Sigma}$, whose boundary consists of $\Sigma$ and its bottom face. At a point $x\in\partial\Sigma$, the capillary construction records only the normal $\nu(x)$ of $\Sigma$, corresponding to $\tilde{\nu}(x)\in\partial\Cth$. By contrast, the set of classical outer unit normals of $\wh{\Sigma}$ at $x$ is $N_{\wh{\Sigma}}(x)\cap\bbS^n$, where
\eq{
N_{\wh{\Sigma}}(x):=\{a\nu(x)-b e_{n+1}:a,b\geq 0\}
}
is the Euclidean normal cone of $\wh{\Sigma}$ at $x$. Hence, for $k\leq n-1$, $\partial\Sigma$ may contribute to the classical $k$-th area measure through normal directions absent from the capillary construction.

For the top-order case $k=n$, $\partial\Sigma$ has zero $n$-dimensional area, and the bottom face $\wh{\Sigma}\cap\{x_{n+1}=0\}$ contributes to $S_n(\wh{\Sigma},\cdot)$ only through the direction $-e_{n+1}\notin\bbS^n_{\theta}$.
\end{remark}

\begin{theorem}\label{thm:grad-bound}
Let $\theta\in(0,\pi/2)$, and let $\Omega\subset\bbR^n$ be a convex body. Suppose $\Si\subset\ov{\Rplus}$ is a $\theta$-capillary hypersurface such that $-\Si$ is the graph of a convex function $f\in C^1(\Omega)$ vanishing on $\partial\Omega$. Then for every $x'\in\Omega$,
\eq{
|Df(x')|\leq \tan\theta, \quad |f(x')|\leq \dist(x',\partial\Omega)\tan\theta.
}
Moreover, if $\Si$ is even and $H:=\max_{x\in\Si}\langle x,e_{n+1}\rangle=\|f\|_{C^0(\Omega)}$, then
\eq{
B_{\frac{H}{\tan\theta}}(0)\subset\Omega,\quad\widehat{\cC}_{\theta,\frac{H}{\tan\theta\sin\theta}}\subset\wh{\Si}.
}
\end{theorem}
\begin{proof}
We have
\eq{
-\Si=\{(x',f(x')):x'\in\Om\},\quad f\leq 0,\quad f=0\ \text{on }\partial\Om.
}
At a boundary point $x'_0\in\partial\Om$, the upward unit normal of the graph of $f$ is
\eq{
\nu(x_0')=\frac{1}{\sqrt{1+|Df(x'_0)|^2}}\bigl(-Df(x'_0),1\bigr).
}
Moreover, by the capillary condition $\langle\nu,e_{n+1}\rangle=\cos\theta$, we have
\eq{
\frac{1}{\sqrt{1+|Df|^2}}=\cos\theta \implies |Df(x'_0)|=\tan\theta \quad\text{ for } x'_0\in\partial\Om.
}
Note that since $f$ is convex and $\Omega$ is bounded and convex, the maximum of $|Df|$ over $\Om$ is attained on $\partial\Om$. Therefore,
\eq{
\label{eq:boundary-Df}
|Df|\leq \tan\theta \quad\text{in }\Om.
}

Let $x'\in\op{int}(\Om)$ and take $y'\in\partial\Om$ such that $|x'-y'|=\dist(x',\partial\Om)$. Also set
\eq{
\xi:=\frac{x'-y'}{|x'-y'|},\quad \gamma(t):=f\bigl(y'+t\xi\bigr),\quad t\in[0,|x'-y'|].
}
Then $\gamma$ is convex, $\gamma(0)=f(y')=0$ and $\gamma(|x'-y'|)=f(x')\leq 0$. Moreover, by \eqref{eq:boundary-Df}, we have $|\gamma'(t)|\leq \tan\theta$. Hence,
\eq{
|f(x')|\leq \int_0^{|x'-y'|} |\gamma'(t)|\,dt\leq \dist(x',\partial\Om)\tan\theta.
}

Now assume that $\Si$ is even. Then $f$ is even, i.e. $f(-x')=f(x')$ for all $x'\in\Om$, and $\Om$ is origin-symmetric. Convexity and evenness imply $f(0)=\min_{\Om} f=-H$. Evaluating $|f(x')|\leq \tan\theta\,\dist(x',\partial\Om)$ at $x'=0$ yields $H\leq \dist(0,\partial\Om)\tan\theta$. Therefore, $B_{\frac{H}{\tan\theta}}(0)\subset\Om$.

To prove the last claim, consider the ($\theta$-capillary) cone in $\bbR^{n+1}$ with apex at $(0,-H)$ and base $B_{\frac{H}{\tan\theta}}(0)\subset\Omega$:
\eq{
\cK^-:=\Bigl\{(x',x_{n+1}): |x'|\leq \frac{H}{\tan\theta},\ -H+|x'|\tan\theta\leq x_{n+1}\leq 0\Bigr\}.
}
The lateral boundary of $\cK^-$ is the graph of the function
\eq{
\psi(x')=-H+|x'|\tan\theta\quad\text{on } B_{\frac{H}{\tan\theta}}(0).
}
Using \eqref{eq:boundary-Df} and $f(0)=-H$, for $|x'|\leq H/\tan\theta$ we have
\eq{
f(x')=\int_0^1 \langle Df(tx'), x'\rangle\,dt -H\leq |x'|\tan\theta -H\leq \psi(x') .
}
Therefore, $\cK^-\subset\wh{\Si^-}$, where
\eq{
\wh{\Si^-}:=\{(x',x_{n+1}): x'\in\Om,\ f(x')\leq x_{n+1}\leq 0\}
}
is the region between the graph of $f$ and $\{x_{n+1}=0\}$. Since the cap $-\widehat{\cC}_{\theta,\frac{H}{\tan\theta\sin\theta}}$ is contained in $\cK^-$, we obtain $\widehat{\cC}_{\theta,\frac{H}{\tan\theta\sin\theta}}\subset\wh{\Si}$.
\end{proof}
\section{Non-collapsing Estimates}\label{sec:non-collapsing}
Let $1\leq k<n$, $\theta\in(0,\pi/2)$, $p\in(1,k+1)$ and $q\in[1,p]$. Let $\Si$ be an even, strictly convex, $\theta$-capillary hypersurface whose capillary support function $s> 0$ solves
\eq{
\label{eq:CMK}
s^{1-q}\sigma_k(\tau^{\sharp}[s])=\phi \quad \text{in }\Cth,
}
with the prescribed function $\phi\in C^{\infty}(\Cth)$. Assume that $\phi_0\leq \phi\leq \phi_1$ with the constants $0<\phi_0<1<\phi_1$.

In this section we establish uniform positive upper and lower bounds for $s$. After deriving an upper bound on $s$ and a positive lower bound on $\max_{\Cth} s$, we treat the model case of rotationally symmetric hypersurfaces in Subsection \ref{subsec:rot-sym}, and then prove the general non-collapsing estimate for even hypersurfaces in Subsection \ref{subsec:even}.

\begin{lemma}\label{lem:upper-s-bound}
Let $s$ satisfy \eqref{eq:CMK}. Then there exists a constant $C_0=C_0(n,k,p,\theta,\phi_0,\phi_1)>1$ such that $s\leq C_0$ on $\cC_{\theta}$.
\end{lemma}

\begin{proof}
Throughout the proof, constants depend only on $(n,k,p,\theta,\phi_0,\phi_1)$.

Integrating by parts (cf. \cite[Cor. 2.10]{MWWX25}) and using the Newton--Maclaurin inequality yields
\eq{
\label{eq:IBP-q}
a_k' \int_{\Cth} s^{1+\frac{k-1}{k}(q-1)} \phi^{\frac{k-1}{k}}\leq \int_{\Cth} s \sigma_{k-1}(\tau^{\sharp}[s])=a_k \int_{\Cth} \ell \sigma_k(\tau^{\sharp}[s])=a_k \int_{\Cth} \ell \phi s^{q-1},
}
where $a_k:=\binom{n}{k-1}/\binom{n}{k}$ and $a_k':=\binom{n}{k-1}\binom{n}{k}^{-\frac{k-1}{k}}$. We can rewrite \eqref{eq:IBP-q} as
\eq{
\label{eq:alpha-beta-q}
\int_{\cC_{\theta}} s^{\alpha(q)}\leq C_1 \int_{\cC_{\theta}} s^{\beta(q)},
}
where
\eq{
\beta(q):=q-1, \quad\alpha(q):=1+\frac{k-1}{k}(q-1)=1+\frac{k-1}{k} \beta(q),
}
and $C_1=C_1(n,k,\theta,\phi_0,\phi_1)>1$. Since $q\in[1,p]$ with $1<p<k+1$, we have
\eq{
0\leq \beta(q)\leq p-1<k, \quad 1\leq \alpha(q)\leq 1+\frac{k-1}{k}(p-1)\leq k.
}

Assume $1<q\leq p$. Then $\beta(q)> 0$ and $0<\beta(q)<\alpha(q)$, and by H\"older's inequality,
\eq{
\Bigl(\int_{\cC_{\theta}} s^{\beta(q)}\Bigr)^{\frac{\alpha(q)}{\beta(q)}}\leq |\cC_{\theta}|^{\frac{\alpha(q)}{\beta(q)}-1}\int_{\cC_{\theta}} s^{\alpha(q)}.
}
Combining with \eqref{eq:alpha-beta-q} we obtain
\eq{
\int_{\cC_{\theta}} s^{\beta(q)}\leq |\cC_{\theta}| C_1^{\frac{\beta(q)}{\alpha(q)-\beta(q)}}.
}

Note that
\eq{
\alpha(q)-\beta(q)=1-\frac{1}{k} \beta(q),\quad\frac{\beta(q)}{\alpha(q)-\beta(q)}=\frac{\beta(q)}{1-\beta(q)/k}.
}
Since $\beta(q)\in[0,p-1]$, we have
\eq{
\frac{\beta(q)}{1-\beta(q)/k}\leq E_p:=\frac{k(p-1)}{k+1-p}.
}
Thus for $1<q\leq p$,
\eq{
\int_{\cC_{\theta}} s^{q-1}=\int_{\cC_{\theta}} s^{\beta(q)}\leq |\cC_{\theta}| C_1^{E_p}.
}
Choosing the constant $C_2=C_2(n,k,p,\theta,\phi_0,\phi_1)$ larger if necessary, we have for all $q\in[1,p]$:
\eq{
\int_{\cC_{\theta}} s^{q-1}\leq C_2.
}

Now we return to \eqref{eq:IBP-q} and we obtain
\eq{
a_k' \int_{\Cth} s^{1+\frac{k-1}{k}(q-1)} \phi^{\frac{k-1}{k}}\leq a_k \phi_1\int_{\Cth} s^{q-1}\leq C_3
}
for some $C_3=C_3(n,k,p,\theta,\phi_0,\phi_1)$. Since $\phi\geq \phi_0$, we obtain
\eq{
\label{eq:Lalpha-q-uniform}
\int_{\Cth} s^{\alpha(q)}\leq C_4
}
for all $q\in[1,p]$, with $C_4$ depending only on $(n,k,p,\theta,\phi_0,\phi_1)$.

Since $\alpha(q)\geq 1$, \eqref{eq:Lalpha-q-uniform} also yields a uniform $L^1$ bound for $s$:
\eq{
\int_{\cC_{\theta}} s\leq |\cC_{\theta}|^{1-\frac{1}{\alpha(q)}}\Bigl(\int_{\cC_{\theta}} s^{\alpha(q)}\Bigr)^{\frac{1}{\alpha(q)}}\leq C_5
}
for all $q\in[1,p]$, with $C_5$ depending only on $(n,k,p,\theta,\phi_0,\phi_1)$.

Finally, the argument in the proof of \cite[Lem. 4.6]{HIS25} implies
\eq{
s\leq C_0 \quad\text{in } \cC_{\theta}
}
for all $q\in[1,p]$ with $C_0=C_0(n,k,p,\theta,\phi_0,\phi_1)$.
\end{proof}

\begin{proposition}\label{prop:ratio}
Let $\tilde{s}:=s/\ell$ where $s$ solves \eqref{eq:CMK}. Then
\eq{
\sigma_k\Bigl(g^{-1}\cdot\bigl(\ell\bHess \tilde{s}+\nabla \tilde{s}\otimes \nabla \ell+\nabla \ell\otimes \nabla \tilde{s}+\tilde{s}g\bigr)\Bigr)=(\tilde{s}\ell)^{q-1}\phi\quad\text{in }\Cth
}
and $\bn_{\mu}\tilde{s}=0$ on $\partial\Cth$.
\end{proposition}

\begin{lemma}\label{thm:LowerMax}
Let $s$ solve \eqref{eq:CMK}. Then
\eq{
\label{s-max-lower-bound}
\max_{\Cth}s\geq \left(\frac{\phi_0}{\binom{n}{k}}\right)^{\frac{1}{k+1-p}}(1-\cos\theta)^{\frac{k}{k+1-p}}.
}
\end{lemma}

\begin{proof}
Let $\ze_{\ast}\in\Cth$ be a maximum point of $\tilde{s}=s/\ell$. By \autoref{prop:ratio}, we have $\bn \tilde{s}(\ze_{\ast})=0$ and $\bHess \tilde{s}(\ze_{\ast})\leq 0$ (even if $\ze_{\ast}\in\partial\Cth$, since $\bn_{\mu}\tilde{s}=0$ on $\partial\Cth$). At $\ze_{\ast}$,
\eq{
\tau[s](\ze_{\ast})=\ell(\ze_{\ast})\bHess \tilde{s}(\ze_{\ast})+\tilde{s}(\ze_{\ast}) g\leq \tilde{s}(\ze_{\ast}) g.
}
Hence
\eq{
\sigma_k(\tau^{\sharp}[s](\ze_{\ast}))\leq \sigma_k(\tilde{s}(\ze_{\ast})\id)=\binom{n}{k}\tilde{s}(\ze_{\ast})^k.
}
By \eqref{eq:CMK} and $\phi\geq \phi_0$,
\eq{
\phi_0\leq (\tilde{s}(\ze_{\ast})\ell(\ze_{\ast}))^{1-q}\sigma_k(\tau^{\sharp}[s](\ze_{\ast}))\leq \binom{n}{k}\tilde{s}(\ze_{\ast})^{k+1-q}\ell(\ze_{\ast})^{1-q}.
}
Now \eqref{s-max-lower-bound} follows from $1-\cos\theta\leq \ell$,  $0<\phi_0<1$, and $q\in[1,p]$ with $1<p<k+1$.
\end{proof}

\subsection{Rotationally symmetric hypersurfaces}\label{subsec:rot-sym}

We first consider rotationally symmetric hypersurfaces as a model case; the estimates in this subsection are not used in the proof of \autoref{thm:height-lower-bound}. Writing $-\Sigma$ as the graph of a convex function $f$ on $\Omega\subset\bbR^n$ as in \autoref{thm:grad-bound}, define
\eq{
r_{\mathrm{out}}:=\max_{x'\in\Omega}|x'|,\quad r_{\mathrm{in}}:=\min_{x'\in\partial\Omega}|x'|.
}
Assume $\det D^2 f\geq \Lambda$ in $\Om$ and $f=0$ on $\partial\Om$. Consider the quadratic barrier
\eq{
Q(x')=\frac{\Lambda^{1/n}}{2}(|x'|^2-r_{\mathrm{in}}^2),\quad x'\in\Om.
}
Then $Q\geq f$ on $\partial\Om$ and $\det D^2 Q\leq \det D^2 f$ in $\Omega$. By the comparison principle,
\eq{
\label{ineq: comparison-principle}
Q\geq f\quad\text{in }\Om \quad\implies\quad H\geq \frac{\Lambda^{1/n}}{2}r_{\mathrm{in}}^2,
}
where $H=-\min f=-f(0)$.

Recall that the Gauss curvature of $\Sigma$ is given by
\eq{
\cK=\frac{\det D^2 f}{(1+|Df|^2)^{(n+2)/2}}\quad \text{in }\Omega.
}
Since $\phi=s^{1-q}\sigma_k(\tau^{\sharp}[s])\geq c_k s^{1-q}\sigma_n(\tau^{\sharp}[s])^{k/n}$ in $\Cth$ with $c_k=\binom{n}{k}$, we have
\eq{
\cK\geq c_k^{n/k}\phi^{-n/k}s^{n(1-q)/k}
}
and
\eq{
\label{eq:det-lower-compact}
\det D^2 f\geq c_k^{n/k}\phi_1^{-n/k}(s_{\max})^{\frac{n(1-q)}{k}},
}
where $\phi_0\leq \phi\leq \phi_1$ with the constants $0<\phi_0<1<\phi_1$.

\begin{theorem}\label{thm:H-not-zero}
Let $\Sigma$ be a rotationally symmetric, strictly convex, $\theta$-capillary hypersurface whose capillary support function $s$ satisfies \eqref{eq:CMK}. Then
\eq{
H\geq H_{\star},\quad H_{\star}=H_{\star}(n,k,p,\theta,\phi_0,\phi_1).
}
In particular, $H_{\star}\cos\theta\leq s\leq C_0$.
\end{theorem}

\begin{proof}
By \autoref{lem:upper-s-bound}, $s\leq C_0$. Due to \eqref{eq:det-lower-compact} and \eqref{ineq: comparison-principle}, we have
\eq{
H\geq \frac{c_k^{1/k}}{2}\phi_1^{-1/k}C_0^{\frac{1-q}{k}}r_{\mathrm{in}}^2\geq \frac{c_k^{1/k}}{2}\phi_1^{-1/k}C_0^{\frac{1-p}{k}}r_{\mathrm{in}}^2,
}
where we used that $C_0>1$ and $q\in[1,p]$. Since $\Si$ is rotationally symmetric, $r_{\mathrm{in}}=r_{\mathrm{out}}$ and thus $s\leq r_{\mathrm{in}}+H$. Now, by \autoref{thm:LowerMax} and \autoref{thm:grad-bound},
\eq{
c_0\leq s_{\max}\leq r_{\mathrm{in}}+H\leq (1+\tan\theta)r_{\mathrm{in}},
}
where $c_0=\left(\frac{\phi_0}{\binom{n}{k}}\right)^{\frac{1}{k+1-p}}(1-\cos\theta)^{\frac{k}{k+1-p}}$. Hence
\eq{
r_{\mathrm{in}}\geq \frac{c_0}{1+\tan\theta},
}
and the lower bound on $H$ follows. Due to $s\geq H\cos\theta$, the proof is complete.
\end{proof}

\subsection{Even hypersurfaces}\label{subsec:even}

For general even, strictly convex, $\theta$-capillary hypersurfaces we work directly with area measures. Given solutions to \eqref{eq:CMK} with $\phi=\psi_i$, $q_i\in[1,p]$, and $H_i\to 0$, let $K_i$ be the union of $\widehat{\Sigma}_i$ and its reflection across $\{x_{n+1}=0\}$. By Blaschke's selection theorem, a subsequence converges to a nontrivial limit $K_{\infty}\subset e_{n+1}^{\perp}$ with linear span $L=\operatorname{lin}(K_{\infty})$ of dimension $m=\dim L\in\{1,\dots,n\}$. On any open belt $\cB\Subset\op{int}(\mathbb{S}^n_{\theta})$ whose closure is disjoint from $\mathbb{S}^n\cap L^{\perp}$, by \eqref{eq:CMK},
\eq{
\binom{n}{k}\int_{\cB} h_i^{1-q_i}\,dS_k(K_i,u)=\int_{T^{-1}\cB}\psi_i,
}
while \autoref{thm:GKW}, \autoref{rem:Sk-vanish} and \autoref{cor:subspace-support} describe $S_k(K_{\infty},\cdot)$ on $\cB$. When $m\leq k$, since $\psi_i\geq \phi_0$, the right-hand side is at least $\phi_0\cH^n(\cB)> 0$, contradicting $S_k(K_i,\cB)\to 0$. When $m\geq k+1$, since $\psi_i\leq \phi_1$, the right-hand side is at most $\phi_1\cH^n(\cB)$, contradicting \autoref{thm:GKW} as $\cB$ approaches $\mathbb{S}^n\cap L^{\perp}$.

We also mention \cite{PS24}, where lower-dimensional Hausdorff limits of area measures were analyzed for the closed $L_p$-Christoffel--Minkowski problem on $\mathbb{S}^n$. In the capillary setting, the union $K$ of $\widehat{\Sigma}$ and its reflection across $\{x_{n+1}=0\}$ has nontrivial normal cones along $\partial\Sigma\subset\{x_{n+1}=0\}$ that may contribute to $S_k(K,\cdot)$ outside $\op{int}(\mathbb{S}^n_{\theta})\cup\op{int}(-\mathbb{S}^n_{\theta})$ (see \autoref{area measure-standard-vs-capillary}). In our argument, we avoid any contribution from $\partial\Sigma$ by restricting $S_k(K,\cdot)$ to open belts $\cB\Subset\op{int}(\mathbb{S}^n_{\theta})$.

\begin{theorem}[\cite{GKW11}, Thm. 6.2]\label{thm:GKW}
Let $L\subset\mathbb{R}^{n+1}$ be an $m$-dimensional linear subspace with $2\leq m\leq n$, $K\subset L$ be a convex body in $L$, and  $k\in\{1,\dots,m-1\}$. Then, for every nonnegative measurable function $\psi$ on $\mathbb{S}^n$,
\eq{
\int_{\mathbb{S}^n} \psi(u) \, dS_k(K, u)=c_{m,k} \int_{\mathbb{S}^n \cap L} I(\xi) \, dS_k^L(K, \xi),
}
where
\eq{
I(\xi):=\int_{\mathbb{S}^n \cap L^{\perp}} \int_0^{\pi/2} \psi(\sin\beta\,\xi+\cos\beta\,\eta) \sin^{m-k-1}\beta \, \cos^{n-m}\beta \, d\beta \, d\cH^{n-m}(\eta),
}
and
\eq{
c_{m,k}:=\frac{\binom{m-1}{k}}{\binom{n}{k}}.
}
\end{theorem}

\begin{proof}
The integral formulation follows directly from \cite[Thm. 6.2]{GKW11}, which states:
\eq{
\binom{m-1}{k} \pi_{L,-k}^* S_k^L(K, \cdot)=\binom{n}{k} S_k(K, \cdot).
}
By the definition of the lifting operator $\pi_{L,-k}^*$ (cf. \cite[Def. 5.2]{GKW11}):
\eq{
\pi_{L,-k}^* S_k^L(K, A)=\int_{\mathbb{S}^n \cap L} \int_{H^{n+1-m}(L,\xi) \cap A} \langle\xi, w\rangle^{m-k-1} dw \, S_k^L(K, d\xi).
}
In our coordinates, $w=\cos\beta\,\eta+\sin\beta\,\xi$, which implies $\langle\xi, w\rangle=\sin\beta$. Moreover, on the relatively open $(n+1-m)$-dimensional half-sphere
\eq{
H^{n+1-m}(L,\xi):=\bigl\{w\in\mathbb{S}^n\setminus L^{\perp}: \operatorname{pr}_L(w)=\xi\bigr\},
}
we have $dw:=d\cH^{n+1-m}(w)=\cos^{n-m}\beta \, d\beta \, d\cH^{n-m}(\eta)$. Here, $\operatorname{pr}_L(w)$ is the spherical projection of $w$ on $\bbS^n\cap L$.
\end{proof}

\begin{lemma}\label{lem:belt-lower-estimate}
Let $L\subset\mathbb{R}^{n+1}$ be a linear subspace with $\dim L=m\in\{1,\dots,n\}$, and let $K\subset L$ be a convex body (with nonempty interior) in $L$. Suppose $k\in\{1,\dots,m-1\}$. Let $\mathcal{U}\subset\mathbb{S}^n\cap L$ and $\mathcal{V}\subset\mathbb{S}^n\cap L^{\perp}$ be relatively open sets with
\eq{
S_k^L(K,\mathcal{U})> 0\quad\text{and}\quad \cH^{n-m}(\mathcal{V})> 0.
}
For angles $0<\beta_1<\beta_2<\pi/2$, define the belt
\eq{
\mathcal{B}=\bigl\{u=\sin\beta\,\xi+\cos\beta\,\eta:\ \eta\in\mathcal{V},\ \xi\in\mathcal{U},\ \beta\in(\beta_1,\beta_2)\bigr\}\subset\mathbb{S}^n.
}
Then
\eq{
S_k(K,\cB)=c_{m,k}\cH^{n-m}(\mathcal{V})S_k^L(K,\mathcal{U})\int_{\beta_1}^{\beta_2} \sin^{m-k-1}\beta\,\cos^{n-m}\beta\,d\beta.
}
\end{lemma}

\begin{proof}
The claim follows from \autoref{thm:GKW} with the choice $\psi=\mathbf{1}_{\mathcal{B}}$.
\end{proof}

\begin{lemma}\label{rem:Sk-vanish}
Let $K\subset\bbR^{n+1}$ be a nonempty compact convex set and $L=\op{lin}(K)$. If $k>m=\dim L$, then $S_k(K,\cdot)\equiv 0$.
\end{lemma}

\begin{proof}
For a Borel set $\omega\subset\bbS^n$ and $\rho> 0$ define
\eq{
B_{\rho}(K,\omega)=\{x\in\bbR^{n+1}: 0<d(K,x)\leq \rho,\, u(K,x)\in\omega\},
}
where $d(K,x)$ is the Euclidean distance from $x$ to $K$, $p(K,x)$ is the nearest point of $K$ to $x$, and
\eq{
u(K,x):=\frac{x-p(K,x)}{\abs{x-p(K,x)}}.
}

By the local Steiner formula (cf. \cite[(4.13)]{Sch14}),
\eq{
\label{eq:local-Steiner}
\cH^{n+1}\bigl(B_{\rho}(K,\omega)\bigr)=\frac{1}{n+1}\sum_{j=0}^n\binom{n+1}{j}\rho^{n+1-j}S_j(K,\omega).
}

Since $K\subset L$, we have
\eq{
\{x\in\bbR^{n+1}: d(K,x)\leq \rho\}\subset(K+\rho B_L)+\rho B_{L^{\perp}},
}
where $B_L=\bbB\cap L$ and $B_{L^{\perp}}=\bbB\cap L^{\perp}$ are the unit balls in $L$ and $L^{\perp}$, respectively. In particular, for $\rho\leq 1$:
\eq{
\label{eq:upper}
\cH^{n+1}\bigl(B_{\rho}(K,\omega)\bigr)\leq \cH^m(K+\rho B_L)\cH^{n+1-m}(\rho B_{L^{\perp}})\leq C\rho^{n+1-m},
}
where $C:=\cH^m(K+B_L)\cH^{n+1-m}(B_{L^{\perp}})$.

On the other hand, if $S_k(K,\omega)> 0$ for some Borel set $\omega$, then \eqref{eq:local-Steiner} yields
\eq{
\label{eq:lower}
\cH^{n+1}\bigl(B_{\rho}(K,\omega)\bigr)\geq c\rho^{n+1-k},\quad c=c(n,K,\omega)> 0.
}
Combining \eqref{eq:upper} and \eqref{eq:lower}, we obtain
\eq{
c\rho^{n+1-k}\leq C\rho^{n+1-m}\quad \text{for all }0<\rho\leq 1.
}
Since $k>m$, letting $\rho\to 0^{+}$ yields a contradiction.
\end{proof}

\begin{lemma}[\cite{Sch14}, (4.20), (4.24)]\label{lem:polytope-any-dim}
Let $d\geq 2$ and let $P\subset\mathbb{R}^d$ be a (not necessarily full-dimensional) convex polytope. For $k\in\{0,1,\dots,d-1\}$ and every Borel set $\omega\subset\Sd$,
\eq{
S_k(P,\omega)=\sum_{F\in\cF_k(P)} \frac{\Hd^{d-1-k}\bigl(N(P,F)\cap \omega\bigr)}{\binom{d-1}{k}}\Hd^k(F).
}
Here $\cF_k(P)$ is the set of $k$-faces of $P$, $N(P,F)$ is the normal cone of $P$ at $F$ (i.e. the set of all outer normal vectors of $P$ at any $x\in\operatorname{relint}F$ together with the zero vector), and $S_k(P,\cdot)$ is the $k$-th area measure of $P$ on $\Sd$. In particular,
\eq{
\supp S_k(P,\cdot)\subset\bigcup_{F\in\cF_k(P)} \bigl(N(P,F)\cap \Sd\bigr)=\bigcup_{F\in\cF_k(P)} \nu_P(\op{relint}(F)),
}
where $\nu_P$ denotes the spherical image of $P$.
\end{lemma}

\begin{corollary}\label{cor:subspace-support}
Suppose $L\subset\mathbb{R}^d$ is a linear subspace with $m=\dim L\in\{1,\dots,d-1\}$, and let $P\subset L$ be an $m$-dimensional polytope. Then $S_m(P,\cdot)$ is concentrated on $\Sd\cap L^{\perp}$:
\eq{
S_m(P,\omega)=\frac{\Hd^{d-1-m}\bigl(L^{\perp}\cap \omega\bigr)}{\binom{d-1}{m}}\Hd^m(P),\quad\supp S_m(P,\cdot)\subset\Sd\cap L^{\perp}.
}
\end{corollary}

\begin{proof}
By \autoref{lem:polytope-any-dim} with $k=m$, the only $m$-face is $P$ and $N(P,P)=L^{\perp}$.
\end{proof}

\begin{lemma}\label{lem:belt-control}
Let $K_i\subset\bbR^{n+1}$ be a sequence of origin-symmetric convex bodies with $K_i\to K_{\infty}$ in the Hausdorff metric and assume that $K_{\infty}\subset e_{n+1}^{\perp}$ is not a single point. Let
\eq{
L:=\op{lin}(K_{\infty})\subset e_{n+1}^{\perp},\quad m:=\dim L\in\{1,\dots,n\},\quad\mathcal{U}:=\bbS^n\cap L.
}
Then there exists a constant $c_{\star}> 0$ with the following property: for any angles $0<\beta_1<\beta_2<\theta$, there exist $i_0\in\bbN$ and a relatively open spherical cap $\mathcal{V}\subset\bbS^n\cap L^{\perp}$ centered at $e_{n+1}$ such that the open belt
\eq{
\mathcal{B}:=\bigl\{u=\sin\beta\,\xi+\cos\beta\,\eta:\ \xi\in\mathcal{U},\ \eta\in\mathcal{V},\ \beta\in(\beta_1,\beta_2)\bigr\}\subset\bbS^n
}
satisfies:
\begin{enumerate}[(i)]
\item $\mathcal{B}\Subset \op{int}(\bbS^n_{\theta})$ and $\overline{\mathcal{B}}\cap(\bbS^n\cap L^{\perp})=\emptyset$;
\item for all $i\geq i_0$ and all $u\in\overline{\mathcal{B}}$,
\eq{
\label{eq:belt-LB}
h_{K_i}(u)\geq c_{\star}\,\sin\beta_1.
}
\end{enumerate}
\end{lemma}

\begin{proof}
Write $h_i:=h_{K_i}$ and $h_{\infty}:=h_{K_{\infty}}$. Since $K_{\infty}$ has nonempty interior in $L$, there exists $c_{\star}> 0$ such that
\eq{
h_{\infty}(\xi)\geq 4c_{\star}\quad\forall\,\xi\in\mathcal{U}.
}
By the uniform convergence of $h_i\to h_{\infty}$, there exists $i_0$ such that for all $i\geq i_0$,
\eq{
\label{eq:h-lb-U}
h_i(\xi)\geq 2c_{\star}\quad\forall\,\xi\in\mathcal{U}.
}

Since $K_{\infty}\subset L$, we have $h_{\infty}(\eta)=0$ for every $\eta\in\bbS^n\cap L^{\perp}$. Let $0<\beta_1<\beta_2<\theta$. Choose $\epsilon> 0$ with $\epsilon<\theta-\beta_2$ and define the spherical cap $\mathcal{V}\subset\bbS^n\cap L^{\perp}$ by
\eq{
\mathcal{V}:=\{\eta\in\bbS^n\cap L^{\perp}: \angle(\eta,e_{n+1})<\epsilon\}.
}
Since $(\xi,\eta,\beta)\mapsto\sin\beta\,\xi+\cos\beta\,\eta$ is a diffeomorphism from $(\bbS^n\cap L)\times(\bbS^n\cap L^{\perp})\times(0,\pi/2)$ onto $\bbS^n\setminus(L\cup L^{\perp})$, the set $\mathcal{B}$ is open in $\bbS^n$. Moreover, for any $\eta\in\overline{\mathcal{V}}$ and any $\beta\in[\beta_1,\beta_2]$:
\eq{
\langle\sin\beta\,\xi+\cos\beta\,\eta, e_{n+1}\rangle=\cos\beta\,\langle\eta,e_{n+1}\rangle\geq \cos\beta\,\cos\epsilon\geq \cos(\beta+\epsilon)>\cos\theta,
}
which implies $\overline{\mathcal{B}}\subset\op{int}(\bbS^n_{\theta})$. Also, since $\beta\geq \beta_1>0$, the set $\overline{\mathcal{B}}$ is disjoint from $\bbS^n\cap L^{\perp}$.

Next, since $h_{\infty}\equiv 0$ on $\bbS^n\cap L^{\perp}$, uniform convergence of $h_i\to h_{\infty}$ implies (after increasing $i_0$ if necessary) that for all $i\geq i_0$,
\eq{
\label{eq:h-small-V}
\sup_{\eta\in\overline{\mathcal{V}}} h_i(\eta)\leq c_{\star}\tan\beta_1.
}

Let $i\geq i_0$ and $u\in\overline{\mathcal{B}}$. Then $u=\sin\beta\,\xi+\cos\beta\,\eta$ for some $\xi\in\mathcal{U}$, $\eta\in\overline{\mathcal{V}}$ and $\beta\in[\beta_1,\beta_2]$. Choose $x_i\in K_i$ with $\langle x_i,\xi\rangle=h_i(\xi)$. Since $K_i$ is origin-symmetric, we have $\langle x_i,\eta\rangle\geq -h_i(\eta)$. Hence,
\eq{
h_i(u)\geq \langle x_i,u\rangle=\sin\beta\,h_i(\xi)+\cos\beta\,\langle x_i,\eta\rangle\geq \sin\beta\,h_i(\xi)-\cos\beta\,h_i(\eta).
}
Using \eqref{eq:h-lb-U}, \eqref{eq:h-small-V}, $\sin\beta\geq \sin\beta_1$ and $\cos\beta\leq \cos\beta_1$, we obtain
\eq{
h_i(u)\geq \sin\beta_1\,(2c_{\star})-\cos\beta_1\left(c_{\star}\tan\beta_1\right)=c_{\star}\sin\beta_1,
}
which proves \eqref{eq:belt-LB}.
\end{proof}

\begin{theorem}\label{thm:height-lower-bound}
Suppose $\Sigma$ is an even, strictly convex, $\theta$-capillary hypersurface whose capillary support function $s$ satisfies \eqref{eq:CMK}. Then
\eq{
H=\max_{x\in\Si}\langle x,e_{n+1}\rangle\geq H_{\star}>0, \quad H_{\star}\cos\theta\leq s\leq C_0
}
with $H_{\star}=H_{\star}(n,k,p,\theta,\phi_0,\phi_1,C_0)$.
\end{theorem}

\begin{proof}
Let $K$ denote the union of $\widehat{\Sigma}$ and its reflection across the hyperplane $\{x_{n+1}=0\}$ and set $h:=h_K$. Assume for contradiction that there exist a sequence $(q_i,\psi_i,\Sigma_i,s_i,K_i,h_i)$ satisfying \eqref{eq:CMK} with $\phi=\psi_i$, $q_i\in[1,p]$ and $\phi_0\leq \psi_i\leq \phi_1$, while
\eq{
H_i:=s_i((1-\cos\theta)e_{n+1})\to 0,\quad q_i\to q_{\ast}\in[1,p].
}
Note that by \autoref{lem:upper-s-bound}, we have
\eq{
\sup_{\Cth} s_i\leq C_0 \quad\text{for all }i.
}
In view of \cite[Lem. 4.2]{HIS25} and the Blaschke selection theorem, after passing to a subsequence, $K_i\to K_{\infty}$ in the Hausdorff metric. By \autoref{thm:LowerMax}, the origin-symmetric limit $K_{\infty}\subset e_{n+1}^{\perp}$ is not a single point.

Let $L:=\operatorname{lin}(K_{\infty})$ and $m:=\dim L\in\{1,\dots,n\}$, and let $c_{\star}> 0$ be the constant in \autoref{lem:belt-control}. For any angles $0<\beta_1<\beta_2<\theta$ with $c_{\star}\sin\beta_1< 1$, by \autoref{lem:belt-control} there exist an index $i_0\in\bbN$, a cap $\mathcal{V}\subset\bbS^n\cap L^{\perp}$, and an open belt $\mathcal{B}\Subset\operatorname{int}(\mathbb{S}^n_{\theta})$ such that for all $i\geq i_0$ and all $u\in\overline{\cB}$,
\eq{
h_i(u)\geq c_{\star}\sin\beta_1.
}
Since $c_{\star}\sin\beta_1< 1$ and $q_i\in[1,p]$, we obtain on $\overline{\cB}$:
\eq{
\label{eq:belt-weights}
C_0^{1-p}\leq h_i^{1-q_i}(u)\leq (c_{\star}\sin\beta_1)^{1-p}\quad\text{for all }u\in\overline{\cB},\ i\geq i_0.
}

Next, we show that $m\geq k$. Suppose $m<k$. By \autoref{rem:Sk-vanish}, we have $S_k(K_{\infty},\cdot)\equiv 0$, and in particular $S_k(K_{\infty},\overline{\cB})=0$. Since $\overline{\cB}$ is closed and $S_k(K_i,\cdot)\to S_k(K_{\infty},\cdot)$ weakly, it follows that
\eq{
\limsup_{i\to \infty} S_k(K_i,\cB)\leq \limsup_{i\to \infty} S_k(K_i,\overline{\cB})\leq S_k(K_{\infty},\overline{\cB})=0.
}
By \eqref{eq:belt-weights}, we obtain
\eq{
\int_{\cB} h_i^{1-q_i}\,dS_k(K_i,u)\leq (c_{\star}\sin\beta_1)^{1-p} S_k(K_i,\cB)\to 0.
}
On the other hand, we have
\eq{
\binom{n}{k}\int_{\cB} h_i^{1-q_i}\,dS_k(K_i,u)=\int_{T^{-1}\cB}\psi_i\geq \phi_0 \cH^n(\cB)>0,
}
a contradiction. Thus $m\geq k$.

\textit{Case 1: $m\geq k+1$.} Recall $K_{\infty}$ has non-empty interior in $L$, which implies $S^L_k(K_{\infty},\mathcal{U})> 0$ for $\mathcal{U}=\mathbb{S}^n\cap L$. By \autoref{lem:belt-lower-estimate},
\eq{
S_k(K_{\infty},\cB)=c_{m,k}\cH^{n-m}(\mathcal{V})S_k^L(K_{\infty},\mathcal{U})\int_{\beta_1}^{\beta_2}\sin^{m-k-1}\beta\,\cos^{n-m}\beta\,d\beta.
}
Since $\cB$ is open in $\bbS^n$ and $S_k(K_i,\cdot)\to S_k(K_{\infty},\cdot)$ weakly, we have
\eq{
\liminf_{i\to \infty} S_k(K_i,\cB)\geq S_k(K_{\infty},\cB).
}
Using \eqref{eq:belt-weights}, we obtain for $i\geq i_0$,
\eq{
\int_{\cB} h_i^{1-q_i}\,dS_k(K_i,u)\geq C_0^{1-p} S_k(K_i,\cB).
}
Passing to the limit as $i\to \infty$, we find
\eq{
&\liminf_{i\to \infty}\int_{\cB} h_i^{1-q_i}\,dS_k(K_i,u)\\
&\geq C_0^{1-p} S_k(K_{\infty},\cB)\\
&=C_0^{1-p} c_{m,k}\cH^{n-m}(\mathcal{V})S_k^L(K_{\infty},\mathcal{U})\int_{\beta_1}^{\beta_2}\sin^{m-k-1}\beta\,\cos^{n-m}\beta\,d\beta.
}
On the other hand, we have
\eq{
\binom{n}{k}\int_{\cB} h_i^{1-q_i}\,dS_k(K_i,u)&=\int_{T^{-1}\cB}\psi_i\\
&\leq \phi_1\cH^n(\cB)\\
&=\phi_1\cH^{m-1}(\mathcal{U})\cH^{n-m}(\mathcal{V})\int_{\beta_1}^{\beta_2}\sin^{m-1}\beta\,\cos^{n-m}\beta\,d\beta.
}
Since the right-hand side is independent of $i$, we have
\eq{
&\binom{n}{k}\limsup_{i\to \infty}\int_{\cB} h_i^{1-q_i}\,dS_k(K_i,u)\\
&\leq \phi_1\cH^{m-1}(\mathcal{U})\cH^{n-m}(\mathcal{V})\int_{\beta_1}^{\beta_2}\sin^{m-1}\beta\,\cos^{n-m}\beta\,d\beta.
}
Combining the upper and lower bounds and canceling the common factor $\cH^{n-m}(\mathcal{V})> 0$,
\eq{
\frac{\binom{m-1}{k} C_0^{1-p}}{\phi_1} S_k^L(K_{\infty},\mathcal{U})\leq \cH^{m-1}(\mathcal{U})\frac{\int_{\beta_1}^{\beta_2}\sin^{m-1}\beta\,\cos^{n-m}\beta\,d\beta}{\int_{\beta_1}^{\beta_2}\sin^{m-k-1}\beta\,\cos^{n-m}\beta\,d\beta}.
}
Since $0<\beta_1<\beta_2<\theta$ with $c_{\star}\sin\beta_1< 1$ are arbitrary, $\mathcal{U}$ and $K_{\infty}$ are independent of $\beta_1,\beta_2$, and the integrands are continuous and positive on $(0,\theta)$, letting $\beta_2\to \beta_1^{+}$ yields
\eq{
\frac{\binom{m-1}{k} C_0^{1-p}}{\phi_1} S_k^L(K_{\infty},\mathcal{U})\leq \cH^{m-1}(\mathcal{U})(\sin\beta_1)^k.
}
Letting $\beta_1\to 0^{+}$, we find $S_k^L(K_{\infty},\mathcal{U})\leq 0$, a contradiction.

\textit{Case 2: $m=k$.} Approximate $K_{\infty}$ in $L$ by $m$-dimensional polytopes $P_j\subset L$. In view of \autoref{cor:subspace-support}, we have $\supp S_k(P_j,\cdot)\subset\mathbb{S}^n\cap L^{\perp}$ for every $j$, and weak convergence $S_k(P_j,\cdot)\to S_k(K_{\infty},\cdot)$ implies $\supp S_k(K_{\infty},\cdot)\subset\mathbb{S}^n\cap L^{\perp}$. By \autoref{lem:belt-control}(i), the closed set $\ov{\cB}$ is disjoint from $\mathbb{S}^n\cap L^{\perp}$, and thus $S_k(K_{\infty},\ov{\cB})=0$. Weak convergence $S_k(K_i,\cdot)\to S_k(K_{\infty},\cdot)$ then yields
\eq{
\limsup_{i\to \infty} S_k(K_i,\cB)\leq \limsup_{i\to \infty} S_k(K_i,\ov{\cB})\leq S_k(K_{\infty},\ov{\cB})=0.
}
By \eqref{eq:belt-weights}, we have
\eq{
\int_{\cB} h_i^{1-q_i}\,dS_k(K_i,u)\leq (\sup_{\cB} h_i^{1-q_i})S_k(K_i,\cB)\leq (c_{\star}\sin\beta_1)^{1-p} S_k(K_i,\cB)\to 0.
}
On the other hand,
\eq{
\binom{n}{k}\int_{\cB} h_i^{1-q_i}\,dS_k(K_i,u)=\int_{T^{-1}\cB}\psi_i\geq \phi_0\cH^n(\cB)>0,
}
a contradiction.

Thus in all cases our assumption $H_i\to 0$ leads to a contradiction. Therefore there exists $H_{\star}> 0$, depending only on $(n,k,p,\theta,\phi_0,\phi_1,C_0)$, such that $H\geq H_{\star}$ for every even, strictly convex, $\theta$-capillary solution of \eqref{eq:CMK} with $q\in[1,p]$ and $\phi_0\leq \phi\leq \phi_1$.

Finally, since $\Si$ is even, we have $H_{\star} e_{n+1}\in\widehat{\Sigma}$. Hence, $s\geq H_{\star}\cos\theta$.
\end{proof}

\section{Regularity Estimates}\label{sec:regularity}
Throughout this section, let $1\leq k<n$, $\theta\in(0,\pi/2)$, $p\in(1,k+1)$, $q\in[1,p]$, and $0<\phi\in C^{\infty}(\Cth)$ with $\phi_0\leq \phi\leq \phi_1$ for constants $0<\phi_0<1<\phi_1$.

\begin{lemma}\label{lem:curv-radii-p}
Let $\Sigma$ be an even, strictly convex, $\theta$-capillary hypersurface whose capillary support function $s$ satisfies \eqref{eq:CMK}. Then
\eq{
\sigma_1(\tau^{\sharp}[s])\leq C\quad\text{in }\cC_{\theta},
}
for some constant $C$ depending only on $n,k,p,\theta,\phi$.
\end{lemma}

\begin{proof}
By \eqref{eq:CMK} and \autoref{lem:upper-s-bound}, for the case $k=1$ we have
\eq{
\sigma_1(\tau^{\sharp}[s])=s^{q-1}\phi\leq C_0^{p-1}\|\phi\|_{C^0}.
}
Now let $2\leq k<n$, and set $F=\sigma_k^{\frac{1}{k}}$. Then $F(\tau^{\sharp}[s])=s^{\frac{q-1}{k}}\phi^{\frac{1}{k}}$. Using the identity
\eq{
\nabla^2_{ii}\si_1=\De \tau_{ii}-n\tau_{ii}+\sigma_1
}
and the concavity of $F$, we have
\eq{
F^{ij}g_{ij}\si_1\leq F^{ij}\nabla^2_{ij}\si_1+ns^{\frac{q-1}{k}}\phi^{\frac{1}{k}}-\De(s^{\frac{q-1}{k}}\phi^{\frac{1}{k}}).
}
We calculate
\eq{
-k\De(s^{\frac{q-1}{k}}\phi^{\frac{1}{k}})&=(1-q)s^{\frac{q-1}{k}-1}\phi^{\frac{1}{k}}\si_1-n(1-q)s^{\frac{q-1}{k}}\phi^{\frac{1}{k}}\\
&\quad+\frac{1}{k}(q-1)(k+1-q)s^{\frac{q-1}{k}-2}|\nabla s|^2 \phi^{\frac{1}{k}}\\
&\quad+2(1-q)s^{\frac{q-1}{k}-1}\langle\nabla s,\nabla \phi^{\frac{1}{k}}\rangle-k s^{\frac{q-1}{k}} \De \phi^{\frac{1}{k}}.
}
Due to the concavity of $F$, we have $F^{ij}g_{ij}\geq c_k:=\binom{n}{k}^{\frac{1}{k}}$. By \autoref{thm:height-lower-bound},
\eq{
\label{s-bound}
1/C_1\leq s\leq C_1
}
for some constant $C_1>1$ depending only on $n,k,p,\theta,\phi$. By \cite[Lem. 4.8]{HIS25},
\eq{
|\nabla s|\leq \frac{C_1}{\sin\theta}.
}
Hence, if $\si_1$ attains its maximum in the interior of $\cC_{\theta}$, we have
\eq{
c_k\si_1&\leq n C_1\|\phi^{\frac{1}{k}}\|_{C^0}+nC_1\|\phi^{\frac{1}{k}}\|_{C^0}+C_1^2 \Bigl(\frac{C_1}{\sin\theta}\Bigr)^2 \|\phi^{\frac{1}{k}}\|_{C^0}\\
&\quad +2 \frac{C_1^2}{\sin\theta}\|\phi^{\frac{1}{k}}\|_{C^1}+nC_1\|\phi^{\frac{1}{k}}\|_{C^2},
}
where we also used that $q\in[1,p]$ with $1<p<k+1$. Thus we have $\si_1\leq C$ for some constant $C=C(n,k,p,\theta,\phi)$.

Now we need to treat the case that the maximum of $\sigma_1$ is attained at a boundary point, say $p_{\ast}$. For every $e_{\alpha}\in T_{p_{\ast}}\partial\Cth$, we have $\nabla_{e_{\alpha}}\mu=\cot\theta\,e_{\alpha}$. Differentiating $\nabla_{\mu}s=\cot\theta\,s$ along $\partial\Cth$, we obtain $\tau(\mu,e_{\alpha})=e_{\alpha}(\nabla_{\mu}s)-\nabla_{\nabla_{e_{\alpha}}\mu}s=0$. Hence we may choose an orthonormal basis $\{\mu\} \cup \{e_{\alpha}\}_{\alpha\geq 2}$ of eigenvectors of $\tau^{\sharp}[s]$ at $p_{\ast}$ such that $\tau_{ij}=\la_i \de_{ij}$. Moreover, using
\eq{
\label{important identity}
\nabla_{\mu}\tau_{\alpha\beta}=\bigl(\tau_{\mu\mu} g_{\alpha\beta}-\tau_{\alpha\beta}\bigr)\cot\theta,\quad 2\leq \alpha,\beta\leq n,
}
and $\nabla_{\mu}s=\cot\theta\,s$, we obtain at $p_{\ast}$ that
\eq{
\label{second proof}
0\leq F^{\mu\mu}\nabla_{\mu}\sigma_1&\leq \cot\theta\Bigl((n+1)\phi^{\frac{1}{k}}s^{\frac{q-1}{k}}-F^{\mu\mu}\sigma_1-\sum_i F^{ii}\la_{\mu}\Bigr)\\
&\quad+s^{\frac{q-1}{k}}\Bigl(\nabla_{\mu}\phi^{\frac{1}{k}}+\frac{q-1}{k}\cot\theta\,\phi^{\frac{1}{k}}\Bigr)
}
and
\eq{
\label{si1 bound 2nd proof}
\sigma_1&\leq \frac{s^{\frac{q-1}{k}}\max_{\partial\Cth}\bigl|\nabla_{\mu}\phi^{\frac{1}{k}}\bigr|}{\cot\theta\,F^{\mu\mu}}+\Bigl(n+1+\frac{q-1}{k}\Bigr)\frac{s^{\frac{q-1}{k}}\phi^{\frac{1}{k}}}{F^{\mu\mu}}\\
&\leq \frac{C_1\|\phi^{\frac{1}{k}}\|_{C^1}}{\cot\theta\,F^{\mu\mu}}+(n+2)\frac{C_1\|\phi^{\frac{1}{k}}\|_{C^0}}{F^{\mu\mu}};
}
see \cite[proof of Lem. 4.9]{HIS25} for details.

Next we show that $F^{\mu\mu}$ has a uniform positive lower bound. By \eqref{s-bound},
\eq{
\label{dot F mumu estimate}
c_1:=(\min_{\Cth} \phi) C_1^{1-p}&\leq \phi s^{q-1}=\si_k(\la)\\
&=\la_{\mu} \si_{k-1}(\la|\la_{\mu})+\si_k(\la|\la_{\mu})\\
&\leq \la_{\mu} \si_{k-1}(\la|\la_{\mu})+c_1'\si_{k-1}(\la|\la_{\mu})^{\frac{k}{k-1}}\\
&\leq c_2 \la_{\mu}F^{\mu\mu}+c_3(F^{\mu\mu})^{\frac{k}{k-1}},
}
where $c_1'=c_1'(n,k)> 0$ and we used that
\eq{
F^{\mu\mu}&=\frac{1}{k}\si_k^{\frac{1-k}{k}}(\la)\si_{k-1}(\la|\la_{\mu})=\frac{1}{k}(s^{q-1}\phi)^{\frac{1-k}{k}}\si_{k-1}(\la|\la_{\mu})\\
&\geq \frac{1}{k} C_1^{\frac{(p-1)(1-k)}{k}}\|\phi^{\frac{1}{k}}\|_{C^0}^{1-k}\si_{k-1}(\la|\la_{\mu}).
}
Here $c_1,c_2,c_3> 0$ depend only on $n,k,p,\theta,\phi$.

Substituting \eqref{dot F mumu estimate} into \eqref{second proof}, we obtain
\eq{
0\leq F^{\mu\mu}\nabla_{\mu}\sigma_1&\leq \Bigl((n+1)\phi^{\frac{1}{k}}s^{\frac{q-1}{k}}-\sum_iF^{ii}\la_{\mu}\Bigr)\cot \theta\\
&\quad +s^{\frac{q-1}{k}}\Bigl(\nabla_{\mu}\phi^{\frac{1}{k}}+\frac{q-1}{k}\cot\theta\, \phi^{\frac{1}{k}}\Bigr)\\
&\leq \frac{\cot\theta}{c_2}\sum_iF^{ii}\Bigl(-\frac{c_1}{F^{\mu\mu}}+c_3(F^{\mu\mu})^{\frac{1}{k-1}}\Bigr)\\
&\quad +C_1\Bigl(\|\phi^{\frac{1}{k}}\|_{C^1}+(n+2)\cot\theta\, \|\phi^{\frac{1}{k}}\|_{C^0}\Bigr).
}
Since $\sum_i F^{ii}\geq c_k>0$, we see that $F^{\mu\mu}$ has a positive lower bound, and in view of \eqref{si1 bound 2nd proof}, $\sigma_1$ is bounded above by a constant depending only on $n,k,p,\theta,\phi$.
\end{proof}

In view of \autoref{lem:curv-radii-p}, the higher-order regularity follows from \cite{LT86} and Schauder estimates.
\begin{proposition}\label{prop:regularity}
Suppose $\Sigma$ is an even, strictly convex, $\theta$-capillary hypersurface whose capillary support function $s$ satisfies \eqref{eq:CMK}. Then for any $m\geq 1$ we have $\|s\|_{C^m}\leq C_m$ for some constant depending only on $n,k,p,\theta,\phi$.
\end{proposition}

\section{Strict Convexity}\label{sec:strict-convexity}

\begin{theorem}\label{thm:full-rank-p}
Let $\theta\in(0,\pi/2)$, $1\leq k<n$ and $q\geq 1$. Suppose $0<\phi\in C^2(\Cth)$ satisfies
\eq{
\label{eq:phi-structure}
\nabla^2\phi^{-\frac{1}{q+k-1}}+g\phi^{-\frac{1}{q+k-1}}\geq 0\quad\text{in }\Cth,
}
and
\eq{
\label{eq:phi-bdry-p}
\nabla_{\mu}\phi^{-\frac{1}{q+k-1}}\leq \cot\theta \,\phi^{-\frac{1}{q+k-1}}\quad\text{on }\partial \Cth.
}
Let $0\leq s\in C^4(\Cth)$ satisfy $\nabla_{\mu}s=\cot\theta\,s$ on $\partial\Cth$ and $\tau^{\sharp}[s]\geq 0$ in $\Cth$. Suppose that
\eq{
\label{eq:sigma-k-p}
\sigma_k\bigl(\tau^{\sharp}[s]\bigr)=s^{q-1}\phi\quad\text{in }\cC_{\theta}.
}
Denote by $\lambda_1$ the smallest eigenvalue of $\tau^{\sharp}[s]$. If $s> 0$ in $\Cth$, then $\lambda_1> 0$ in $\Cth$.
\end{theorem}

\begin{proof}
The argument is the same as in \cite[Thm. 3.1]{HIS25} for $q=1$. Define
\eq{
F=\sigma_k^{1/k},\quad f=\bigl(s^{q-1}\phi\bigr)^{1/k}.
}
When $\phi$ satisfies \eqref{eq:phi-structure}, we have in the interior of $\Cth$ that
\eq{
L[\lambda_1]:=F^{ij}\nabla^2_{ij}\lambda_1-c\bigl(\lambda_1+\abs{\nabla \lambda_1}\bigr)\leq 0
}
in the viscosity sense. See also \cite[Eq. (3.20)]{CH25}. By the strong minimum principle, if $\lambda_1$ vanishes at an interior point of $\Cth$, then $\lambda_1\equiv 0$ in $\Cth$. Since $s> 0$ in $\Cth$ and $\nabla_{\mu}s=\cot\theta\,s> 0$ on $\partial\Cth$, the function $s$ attains its positive minimum at an interior point $\zeta_0\in\op{int}(\Cth)$. At $\zeta_0$, we have $\nabla^2 s(\zeta_0)\geq 0$ and $\tau[s](\zeta_0)\geq s(\zeta_0)g> 0$, which implies $\lambda_1(\zeta_0)> 0$. Therefore, $\lambda_1> 0$ in $\op{int}(\Cth)$.

Suppose for contradiction that $\lambda_1(p_{\ast})=0$ at some $p_{\ast}\in\partial\Cth$. As in Step 1 of the proof of \cite[Thm. 3.1]{HIS25}, we first show that
\eq{
\label{boundary condition}
\tau_{ii}(p_{\ast})=0\implies\nabla_{\mu} \tau_{ii}(p_{\ast})\geq 0.
}

Choose an orthonormal frame $\{e_i\}_{i=1}^n$ at $p_{\ast}$ such that
\eq{
e_1=\mu,\quad e_{\alpha}\in T_{p_{\ast}}\partial \Cth\quad\text{ for } \alpha=2,\dots,n,
}
with $\tau_{ij}(p_{\ast})=0$ for $i\neq j$.

For $i=\alpha\geq 2$, \eqref{boundary condition} follows from \eqref{important identity}. Now let $i=1$ and assume $\tau_{\mu\mu}(p_{\ast})=0$. Note that \eqref{eq:sigma-k-p} is equivalent to
\eq{
\label{eq:F-eq}
F\bigl(\tau^{\sharp}[s]\bigr)=f\quad\text{in }\Cth.
}

Differentiating \eqref{eq:F-eq} in the $\mu$-direction, we obtain $\sum_i F^{ii}\nabla_{\mu}\tau_{ii}=\nabla_{\mu} f$. Using \eqref{important identity} for $\alpha\geq 2$, we obtain
\eq{
\label{eq:bdry-tau-mu}
F^{\mu\mu}\nabla_{\mu}\tau_{\mu\mu}=\nabla_{\mu} f+\sum_{\alpha\geq 2}F^{\alpha\alpha}(\tau_{\alpha\alpha}-\tau_{\mu\mu})\cot\theta.
}
Since $\tau_{\mu\mu}(p_{\ast})=0$, $1\leq k<n$, and $\sigma_k(\tau^{\sharp}[s](p_{\ast}))=s(p_{\ast})^{q-1}\phi(p_{\ast})>0$, we have $F^{\mu\mu}(p_{\ast})> 0$. By the $1$-homogeneity of $F$,
\eq{
\sum_i F^{ii}\tau_{ii}=F\bigl(\tau^{\sharp}[s]\bigr)=f.
}
Hence, at $p_{\ast}$, $\sum_{\alpha\geq 2}F^{\alpha\alpha}\tau_{\alpha\alpha}=f$. Evaluating \eqref{eq:bdry-tau-mu} at $p_{\ast}$ yields
\eq{
\nabla_{\mu}\tau_{\mu\mu}=\frac{\nabla_{\mu} f+f\cot\theta}{F^{\mu\mu}}.
}
Thus \eqref{boundary condition} for $i=1$ holds provided that
\eq{
\label{eq:f-bdry-cond}
\nabla_{\mu}\log f\geq -\cot\theta\quad\text{on }\partial \Cth.
}

It remains to express \eqref{eq:f-bdry-cond} in terms of $\phi$. Since
\eq{
f=s^{\frac{q-1}{k}}\phi^{\frac{1}{k}},\quad \text{and}\quad\nabla_{\mu}\log s=\cot\theta\quad\text{on }\partial \Cth,
}
we require that
\eq{
\nabla_{\mu}\log\phi\geq -(k+q-1)\cot\theta\quad\text{on }\partial \Cth,
}
which is precisely \eqref{eq:phi-bdry-p}. Let $\gamma$ be the geodesic in $\Cth$ with $\gamma(0)=p_{\ast}$ and $\dot{\gamma}(0)=-\mu(p_{\ast})$. The argument in Steps 2--3 of the proof of \cite[Thm. 3.1]{HIS25} yields $\lambda_1(\gamma(t))\geq \varepsilon t$ for some $\varepsilon> 0$ and all sufficiently small $t> 0$. Now choose $e_i$ with $\tau_{ii}(p_{\ast})=0$ and parallel transport $e_i$ along $\gamma$. Since $\tau(e_i,e_i)\geq \lambda_1$ along $\gamma$, we obtain $(\nabla_{-\mu}\tau)(e_i,e_i)\geq \varepsilon$ at $p_{\ast}$, contradicting \eqref{boundary condition}. Thus $\lambda_1> 0$ on $\Cth$.
\end{proof}

\section{Existence and uniqueness}\label{sec:existence-uniqueness}
Let $k,p,\theta$ and $\phi$ be as in \autoref{thm:main}, and fix $\alpha\in(0,1)$. For $q\in[1,p]$ set
\eq{
\psi_q:=\phi^{\frac{q+k-1}{p+k-1}}.
}
For $(q,s)$ with $q\in[1,p]$ and $s\in C^{4,\alpha}_{\mathrm{even}}(\Cth)$ such that $s> 0$ in $\Cth$ we define
\eq{
\left\{
\begin{aligned}
F(q,s)&=\sigma_k(\tau^{\sharp}[s])-s^{q-1}\psi_q&&\text{in }\Cth,\\
G(q,s)&=\nabla_{\mu}s-\cot\theta\,s&&\text{on }\partial\Cth.
\end{aligned}
\right.
}
If $(F(q,s),G(q,s))=(0,0)$, then $s$ solves
\eq{
\label{eq:t-equation}
\left\{
\begin{aligned}
\sigma_k(\tau^{\sharp}[s])&=s^{q-1} \psi_q&&\text{in }\Cth,\\
\nabla_{\mu} s&=\cot\theta\,s&&\text{on }\partial\Cth.
\end{aligned}
\right.
}
For $q=1$, by \cite[Thm. 1.2]{HIS25} there exists a unique even, smooth, strictly convex, $\theta$-capillary solution $s_1$ of \eqref{eq:t-equation}.

Choose constants $0<\phi_0<1<\phi_1$ such that $\phi_0\leq \phi\leq \phi_1$. Then, for every $q\in[1,p]$,
\eq{
\phi_0\leq \phi_0^{\frac{q+k-1}{p+k-1}}\leq \psi_q\leq \phi_1^{\frac{q+k-1}{p+k-1}}\leq \phi_1.
}

Applying \autoref{lem:upper-s-bound}, \autoref{thm:LowerMax} and \autoref{thm:height-lower-bound} with $\phi$ replaced by $\psi_q$ we obtain constants $C_0,c_0> 0$, independent of $q$, such that every even, smooth solution $s$ of \eqref{eq:t-equation} with $\tau^{\sharp}[s]> 0$ satisfies
\eq{
\label{eq:s-C0-bound}
c_0\leq s\leq C_0 \quad\text{in }\Cth.
}

Moreover, since $\psi_q$ satisfies the structural assumptions of \autoref{thm:full-rank-p}, when $s\in C^{4,\alpha}_{\mathrm{even}}(\Cth)$ is a solution of \eqref{eq:t-equation} with $\tau^{\sharp}[s]\geq 0$ and $s> 0$, we must have
\eq{
\label{eq:CRT-strict}
\tau^{\sharp}[s]> 0 \quad\text{in }\Cth.
}
By elliptic regularity (\cite{LT86} and Schauder estimates), $s\in C^{\infty}(\Cth)$.

Combining \eqref{eq:s-C0-bound} and \eqref{eq:CRT-strict} with \autoref{lem:curv-radii-p} and \autoref{prop:regularity}, we obtain a uniform $C^{4,\alpha}$ bound: there exists $C> 0$ such that
\eq{
\label{eq:C4-a-bound-sol}
\|s\|_{C^{4,\alpha}(\Cth)}\leq C
}
for all even solutions $s\in C^{4,\alpha}_{\mathrm{even}}(\Cth)$ of \eqref{eq:t-equation} with $\tau^{\sharp}[s]\geq 0$ and $s> 0$, uniformly in $q\in[1,p]$.

Define
\eq{
2\delta_0=\min_{\substack{q\in[1,p],\ \zeta\in\Cth}}c_0^{q-1}\psi_q(\zeta).
}
Let $R>C$ and define the bounded open set
\eq{
\cO:=\left\{s\in C_{\mathrm{even}}^{4,\alpha}(\Cth):
\begin{aligned}[t]
&\ \|s\|_{C^{4,\alpha}(\Cth)}<R,\quad 2s>c_0,\\
&\ \tau^{\sharp}[s]>0,\quad\sigma_k\bigl(\tau^{\sharp}[s]\bigr)>\delta_0\quad\text{in }\Cth
\end{aligned}
\right\}.
}
On $\overline{\cO}$ we have $\tau^{\sharp}[s]\geq 0$, $\|s\|_{C^{4,\alpha}(\Cth)}\leq R$, and $\sigma_k(\tau^{\sharp}[s])\geq \delta_0$, which ensures that $F(q,\cdot)$ is uniformly elliptic on $\overline{\cO}$. Moreover, by \eqref{eq:s-C0-bound}, \eqref{eq:CRT-strict}, and \eqref{eq:C4-a-bound-sol},
\eq{
(F(q,s),G(q,s))\neq (0,0)\quad\text{for all }(q,s)\in[1,p]\times\partial\cO.
}

Next we briefly explain how to adapt the definition of degree from \cite{LLN17} to $C^{4,\alpha}_{\mathrm{even}}(\Cth)$. First we identify the cap $\Cth$ with $\ov{\Omega}$, where $\Omega:=B_{\tan(\theta/2)}(0)\subset\bbR^n$, via the translation $T$ followed by stereographic projection from the south pole, and on $\ov{\Omega}$, we consider the reflection map $\cR$ acting by $x\mapsto -x$. Write $E_q(s):=(F(q,s),G(q,s))$, $J$ for the product of the isomorphisms in \cite[(3.1), (3.2)]{LLN17}, and $P_{q,s}$ for the shifted fourth-order linear operator obtained by freezing the coefficients at $s$. The uniform ellipticity of $F$, the uniform obliqueness of $G$, and $\|s\|_{C^{4,\alpha}(\Cth)}\leq R$ allow us to find a shift $N$ such that $P_{q,s}$ is invertible for every $(q,s)\in[1,p]\times\overline{\cO}$. On $[1,p]\times\overline{\cO}$, define
\eq{
K(q,s):=P_{q,s}^{-1}\bigl(JE_q(s)-P_{q,s}s\bigr).
}
In view of the estimates following \cite[(3.8)]{LLN17}, $K$ is jointly continuous and its image is relatively compact in $C^{4,\alpha}(\Cth)$. Now, for each fixed $q\in[1,p]$, define $\Phi_q(s):=s+K(q,s)$ on $\overline{\cO}$. Then $\Phi_q(s)=P_{q,s}^{-1}JE_q(s)$. Moreover, for even $s$, $E_q(s)$ is even and both $J$ and $P_{q,s}$ commute with  pullback via $\mathcal{R}$. Thus $\Phi_q(s)$ is even, and $\Phi_q$ restricts to a compact perturbation of the identity on $C^{4,\alpha}_{\mathrm{even}}(\Cth)$. Since $P_{q,s}^{-1}J$ is an isomorphism and $E_q$ does not vanish on $\partial\cO$, the map $\Phi_q$ does not have any zeros on $\partial\cO$, and thus the Leray--Schauder degree $d(q):=\deg_{\mathrm{LS}}(\Phi_q,\cO,0)$ in $C^{4,\alpha}_{\mathrm{even}}(\Cth)$ is well-defined and independent of $q\in[1,p]$.

At the unique even solution $s_1\in\cO$ for $q=1$, set $\cL:=D_sE_1(s_1)$. Since $E_1(s_1)=0$, we have $D_s\Phi_1(s_1)=P_{1,s_1}^{-1}J\cL$. Moreover, in view of \autoref{lem:L1-iso}, $\cL$ is an isomorphism from $C^{4,\alpha}_{\mathrm{even}}(\Cth)$ onto $C^{2,\alpha}_{\mathrm{even}}(\Cth)\times C^{3,\alpha}_{\mathrm{even}}(\partial\Cth)$; since $J$ and $P_{1,s_1}$ restrict to isomorphisms on the even subspaces, the derivative $D_s\Phi_1(s_1)$ is an isomorphism of $C^{4,\alpha}_{\mathrm{even}}(\Cth)$ onto itself. Thus the Leray--Schauder index of $\Phi_1$ at its unique zero $s_1$ in $\overline{\cO}$ equals $\pm 1$. By excision we have $d(1)=\pm 1$, and by homotopy invariance $d(p)=d(1)\neq 0$. Therefore, there exists $s\in\cO$ such that
\eq{
(F(p,s),G(p,s))=(0,0).
}

Next, we establish uniqueness of solutions to \eqref{capillary-Lp-eq} in the class of even, strictly convex capillary hypersurfaces.

Assume that $s_1$ and $s_2$ are two even, strictly convex capillary solutions to
\eq{
\left\{
\begin{aligned}
\sigma_k(\tau^{\sharp}[s])&=s^{p-1} \phi&&\text{in }\Cth,\\
\nabla_{\mu} s&=\cot\theta\,s&&\text{on }\partial\Cth.
\end{aligned}
\right.
}

By \autoref{def:mixed-volume}, we obtain
\eq{
\label{key-ineq-1}
\int_{\cC_{\theta}} s_2 s_1^{p-1}\phi&=\int_{\cC_{\theta}} s_2\sigma_k(\tau^{\sharp}[s_1])\\
&=(n+1)\binom{n}{k}V\bigl(s_2,\underbrace{s_1,\ldots,s_1}_{k\text{-times}},\ell,\ldots,\ell\bigr)\\
&\geq (n+1)\binom{n}{k}V\bigl(\underbrace{s_1,\ldots,s_1}_{(k+1)\text{-times}},\ell,\ldots,\ell\bigr)^{\frac{k}{k+1}}V\bigl(\underbrace{s_2,\ldots,s_2}_{(k+1)\text{-times}},\ell,\ldots,\ell\bigr)^{\frac{1}{k+1}}\\
&=\Bigl(\int_{\cC_{\theta}} s_1^p \phi\Bigr)^{\frac{k}{k+1}}\Bigl(\int_{\cC_{\theta}} s_2^p \phi\Bigr)^{\frac{1}{k+1}},
}
where we applied the Alexandrov--Fenchel inequality (see \cite[Thm. 3.2]{MWWX25} and also \cite{HWYZ24}):
\eq{
\label{AF-ineq}
V(s_1,s_2,s_3,\ldots,s_{n+1})^2\geq V(s_1,s_1,s_3,\ldots,s_{n+1})V(s_2,s_2,s_3,\ldots,s_{n+1}).
}

On the other hand, by the H\"older inequality we have
\eq{
\label{key-ineq-2}
\int_{\cC_{\theta}} s_2 s_1^{p-1}\phi\leq \Bigl(\int_{\cC_{\theta}} s_2^p \phi\Bigr)^{\frac{1}{p}}\Bigl(\int_{\cC_{\theta}} s_1^p \phi\Bigr)^{\frac{p-1}{p}}.
}
Combining \eqref{key-ineq-1} and \eqref{key-ineq-2} yields
\eq{
\Bigl(\int_{\cC_{\theta}} s_1^p \phi\Bigr)^{\frac{k}{k+1}}\Bigl(\int_{\cC_{\theta}} s_2^p \phi\Bigr)^{\frac{1}{k+1}}\leq \Bigl(\int_{\cC_{\theta}} s_2^p \phi\Bigr)^{\frac{1}{p}}\Bigl(\int_{\cC_{\theta}} s_1^p \phi\Bigr)^{\frac{p-1}{p}}.
}
Rearranging, we find
\eq{
\Bigl(\int_{\cC_{\theta}} s_1^p \phi\Bigr)^{\frac{p-k-1}{p(k+1)}}\geq \Bigl(\int_{\cC_{\theta}} s_2^p \phi\Bigr)^{\frac{p-k-1}{p(k+1)}}.
}
Since $1<p<k+1$, interchanging $s_1$ and $s_2$ yields $\int_{\cC_{\theta}} s_1^p \phi=\int_{\cC_{\theta}} s_2^p \phi$. Thus equality holds in \eqref{key-ineq-1}. By evenness and the equality case in \eqref{AF-ineq}, $s_2=c s_1$ for some $c> 0$; since $p\neq k+1$, it follows that $c=1$ and $s_1=s_2$.

\begin{lemma}\label{lem:L1-iso}
For $m\in\{2,4\}$, $\cL:C^{m,\alpha}_{\mathrm{even}}(\Cth)\to C^{m-2,\alpha}_{\mathrm{even}}(\Cth)\times C^{m-1,\alpha}_{\mathrm{even}}(\partial\Cth)$ is an isomorphism.
\end{lemma}

\begin{proof}
Since $\tau^{\sharp}[s_1]> 0$, the matrix $[a^{ij}]$, $a^{ij}=\sigma_k^{ij}(\tau^{\sharp}[s_1])$, is uniformly positive definite on $\Cth$, and
\eq{
Lv=a^{ij}\nabla^2_{ij}v+\tr(a)v,\quad v\in C^{2,\al}(\Cth)
}
is uniformly elliptic with $C^{\infty}$ coefficients. Let us define the boundary operator
\eq{
Mv=\nabla_{(-\mu)} v+\cot\theta\, v\quad\text{on }\partial\Cth.
}
Note that $\cL(v)=(Lv,-Mv)$.

Using the stereographic projection $\Pi:\Sn\setminus\{-e_{n+1}\}\to \bbR^n$ from the south pole defined by $\Pi(x',x_{n+1})=\frac{x'}{1+x_{n+1}}$ (composed with the translation $T:\Cth\to \bbS^n_{\theta}$), we can rewrite $L$ and $M$ on $\Omega:=B_{\tan(\theta/2)}(0)\subset\bbR^n$ as
\eq{
\begin{cases}
Lu:=a^{ij}(x)D_{ij}u+b^i(x)D_i u+c(x)u&\text{in }\Omega,\\
Mu:=\beta^i(x)D_i u+\cot\theta\, u&\text{on }\partial\Omega,
\end{cases}
}
with $a^{ij},b^i,c\in C^{\infty}(\ov{\Omega})$, $\beta^i\in C^{\infty}(\partial\Omega)$, and $M$ uniformly oblique:
\eq{
\Bigl\langle\beta(x),-\frac{x}{|x|}\Bigr\rangle=\frac{1}{1+\cos\theta}\quad\text{for all }x\in\partial\Omega.
}

Since $s_1$ is even, $(L,M)$ commutes with $\cR$. Moreover, if $v\circ\cR=v$ and $u=v\circ(\Pi\circ T)^{-1}$, then $u(-x)=u(x)$. Consequently, the shifted solution operator constructed in the proof of \cite[Thm. 2.30]{Lie13} preserves evenness. Restricting the same Fredholm argument to the closed even subspaces shows that
\eq{
\cL:C^{2,\alpha}_{\mathrm{even}}(\Cth)\to C^{\alpha}_{\mathrm{even}}(\Cth)\times C^{1,\alpha}_{\mathrm{even}}(\partial\Cth)
}
is Fredholm of index zero.

By \cite[Lem. 5.4]{HIS25}, if $v\in C^2_{\mathrm{even}}(\Cth)$ satisfies
\eq{
Lv=0\quad\text{in }\Cth,\quad Mv=0\quad\text{on }\partial\Cth,
}
then $v\equiv 0$. Thus $\cL$ has trivial kernel on $C^{2,\alpha}_{\mathrm{even}}(\Cth)$, and since its Fredholm index is zero, $\cL$ is an isomorphism for $m=2$. For $m=4$, injectivity follows from the case $m=2$, while higher-order boundary Schauder regularity (bootstrapping \cite[Thm. 2.28]{Lie13}) yields $\cL^{-1}(f,g)\in C^{4,\alpha}_{\mathrm{even}}(\Cth)$ for every $(f,g)\in C^{2,\alpha}_{\mathrm{even}}(\Cth)\times C^{3,\alpha}_{\mathrm{even}}(\partial\Cth)$. By the bounded inverse theorem, $\cL$ is also an isomorphism for $m=4$.
\end{proof}

\section*{Acknowledgments}
Hu was supported by the National Key Research and Development Program of China (Grant No. 2021YFA1001800). Ivaki was supported by the Austrian Science Fund (FWF) under Project P36545. We would like to thank Georg Hofst\"{a}tter for helpful discussions.

\bigskip
\begingroup
\footnotesize
\noindent
\begin{minipage}[t]{0.40\textwidth}
\textsc{School of Mathematical\\ Sciences, Beihang University,\\ Beijing 100191, China}\\
\href{mailto:huyingxiang@buaa.edu.cn}{\texttt{huyingxiang@buaa.edu.cn}}
\end{minipage}
\hfill
\begin{minipage}[t]{0.40\textwidth}
\textsc{Institut f\"{u}r Diskrete\\ Mathematik und Geometrie,\\ Technische Universit\"{a}t Wien,\\ Wiedner Hauptstra{\ss}e 8--10,\\ 1040 Wien, Austria}\\
\href{mailto:mohammad.ivaki@tuwien.ac.at}{\texttt{mohammad.ivaki@tuwien.ac.at}}
\end{minipage}
\par
\endgroup

\end{document}